\documentclass[reqno]{amsart}
\usepackage{hyperref}

\def\N{{\mathbb{N}}}

\def\R{{\mathbb{R}}}

\usepackage{amssymb}
\usepackage{xcolor}
\DeclareUnicodeCharacter{2264}{\leq}
\DeclareUnicodeCharacter{2266}{\leqq}

\theoremstyle{plain}
\newtheorem{theorem}{Theorem}
\newtheorem{proposition}{Proposition}
\newtheorem{definition}{Definition}
\newtheorem{lemma}{Lemma} 
\newtheorem{corollary}{Corollary}
\theoremstyle{remark}
\newtheorem{remark}{Remark}
\newtheorem{example}{Example}

\title[Optimal transport]{Monge-Kantorovich's duality for separable Baire measures in completely regular Hausdorff spaces}
\author{Mohammed Bachir}

\begin{document}
	
	\date{\today} 
	\address{Laboratoire SAMM 4543, Universit\'e Paris 1 Panth\'eon-Sorbonne, France}
	
	\email{Mohammed.Bachir@univ-paris1.fr}
	\begin{abstract}
		We generalize the classical Monge–Kantorovich duality—typically established for tight (Radon) probability measures—to separable Baire probability measures, which are strictly more general than tight measures in general completely regular Hausdorff spaces. Within this broader framework, we also demonstrate the existence of solutions.
	\end{abstract}
	\maketitle
	
	\tableofcontents
	\section{Introduction}

	Given two completely regular Hausdorff spaces $X$ and $Y$, two Radon probability measures $\mu \in \mathcal{P}_r(X)$, $\nu \in \mathcal{P}_r(Y)$ and a cost function (measurable) $c: X\times Y\to [0,+\infty]$, the Kantorovich primal problem is then given by the following minimization problem:
	\begin{eqnarray*} \label{KP}
		{\bf (KP)} \hspace{5mm}  v_{\min}{\bf (KP)}:=\min\left \{\int_{X\times Y} c d\gamma: \gamma \in \Pi(\mu,\nu)\right \},
	\end{eqnarray*}
	where  $\Pi(\mu,\nu)$ is the so-called transport plans, that is, the set of all Radon probability measures $\gamma$ on $X\times Y$ such that $\gamma$ admits $\mu$ and $\nu$ as marginals on $X$ and $Y$ respectively. Equivalently (see \cite[Corollary 2.2]{Eda}),
	\begin{eqnarray*} \Pi(\mu,\nu):= \biggl\{\gamma \in  \mathcal{P}_r(X\times Y) :&&\int_{X\times Y} \phi\oplus \psi d\gamma =\int_X \phi d\mu +\int_Y \psi d\nu,\\
		&& \forall \phi\in C_b(X), \forall \psi \in C_b(Y) \biggr\},
	\end{eqnarray*}
		where $C_b(X)$ (resp. $C_b(Y)$) denotes the space of all real-valued bounded continuous functions on $X$ (resp. on $Y$).
		
	The  Kantorovich dual problem is given by the following maximization problem:
	\begin{eqnarray*} \label{DP}
		{\bf (DP)}	\hspace{5mm} v_{\max}{\bf (DP)}:=\max \left \{ \int_X \phi d\mu+\int_Y \psi d\nu: \phi\in C_b(X), \psi\in C_b(Y); \phi\oplus \psi\leq c \right \},
	\end{eqnarray*}
 where, $\phi\oplus \psi : X\times Y \to \R$ is defined by $\phi\oplus \psi (x,y)=\phi(x)+\psi(y)$, for all $x\in X, y\in Y$. 
	Monge-Kantorovich's duality for optimal transport originated in Monge's transport problem. For details on the  Monge-Kantorovich duality in Polish spaces, we send back the precious books by F. Santambrogio \cite{Sfilippo} and C. Villani \cite{Villani0, Villani}. For the Monge-Kantorovich duality in  the general case of completely regular spaces, we refer to the significant investigations developed by H. G. Kellerer in \cite{Khg} and D.A. Edwards in \cite{Eda}. Several authors have examined the extension of Monge-Kantorovich's duality on various directions, see for instance the works of W. Schachermayer and J. Teichmann \cite{SwTj}, M. Beiglböck and W. Schachermayer \cite{BSch} and the more recent works in \cite{BPS, DKW} and the references therein.
	
	\paragraph{\bf The aim} The concept of tight (or Radon) measure is fundamental in the study of the existence of solutions to the problem {\bf (KP)} (thanks to Prokhorov’s theorem, see for instance \cite{Sfilippo, Villani}). The objective of this article is to extend, the Monge-Kantorovich's duality from tight measures to the case of separable Baire measures in completely regular Hausdorff spaces. A Baire measure $\mu$ on $X$ is called separable if for every bounded continuous pseudo-metric $d$ on $X$, $\mu$ is concentrated in a subset of $X$ which has a countable set dense for $d$ (see Section \ref{S1}, for more details). In a non-complete separable metric space, Borel measures coincide with separable measures but are not necessarily Radon (see \cite[Example 7.1.6 $\&$ Example 1.12.13]{Bvi}). Also, in the context of general completely regular spaces, Radon measures can be extremely rare compared to Borel measures. For example, on the Sorgenfrey line $([0,1), \tau_\ell)$ (a separable perfectly normal Hausdorf space), the Borel sets coincide with the usual Borel sets and every Borel probability measure is separable. However, the Lebesgue measure is not Radon with respect to the Sorgenfrey topology, because (as is well known), compact subsets of $([0,1), \tau_\ell)$ are at most countable (see for instance, \cite[Example 6.1.19]{Bvi}). The results of this article, will complement the investigations conducted by H. G. Kellerer and D.A. Edwards, among other studies, where only the tight (or Radon) measures were studied. 

	In order to carry out our extension, we are also obliged to extend the definition of the set $\Pi(\mu, \nu)$ and consequently the problem {\bf (KP)}. One of the most important reasons, is that, it remains an open problem whether the product of two Borel measures on topological spaces can be always extended to a Borel measure (see comments in \cite[Chapter 7, p. 127]{Bvi}, just before \cite[Theorem ~7.14.11]{Bvi}). This leads us to replace the set $\Pi(\mu, \nu)$ by the broader set $\Gamma(\mu, \nu)$ (see below for definitions) which always exists even if the measures $\mu$ and $\nu$ are not Radon, whereas there is no guarantee that the set $\Pi(\mu, \nu)$ can be nonempty by replacing ``Radon or tight probability'' with ``Baire or Borel probability'' in general completely regular Hausdorff spaces. It should be noted that, the usual product measure of two Baire measures, say $\mu \otimes \nu$, can always be extended into a non-negative finite, additive and regular set function $\widetilde{\mu \otimes \nu} \in \mathcal{M}^1(X\times Y)$ (via the Alexandroff representation and the Hahn-Banach theorems), thus $\widetilde{\mu \otimes \nu} \in \Gamma (\mu, \nu)$ and so $\Gamma(\mu, \nu)$ is nonempty. Thus, we consider in this article the optimization problem ${\bf (OP)}$ as a generalisation of the problem {\bf (KP)}: Given two Baire probability measures $\mu \in \mathcal{P}_\sigma(X)$, $\nu \in \mathcal{P}_\sigma(Y)$ and a cost function $c : X\times Y \to [0,+\infty]$, we define
	\begin{eqnarray*} \label{OP}
		{\bf (OP)} \hspace{5mm} v_{\min}{\bf (OP)}=\inf\left \{\int_{X\times Y} c d\gamma: \gamma \in \Gamma(\mu,\nu)\right \},
	\end{eqnarray*}
	where, 
	\begin{eqnarray*} \Gamma(\mu,\nu):= \biggl\{\gamma \in  \mathcal{M}^1(X\times Y) :&&\int_{X\times Y} \phi\oplus \psi d\gamma =\int_X \phi d\mu +\int_Y \psi d\nu,\\
		 && \forall \phi\in C_b(X), \forall \psi \in C_b(Y) \biggr\},
	\end{eqnarray*}
	and $\mathcal{M}^1(X\times Y)$ denotes the set of all non-negative, finite, additive and regular set functions (charges) $\gamma$ on the algebra $\mathcal{U}(X\times Y)$ generated by the zero-sets, such that $\gamma(X\times Y)=1$ (for more details, see Section \ref{S1}. For integration theory via linear functionals, representing set functions and measures, we refer to \cite[Chapter 7]{Bvi} and \cite[Chapter 14]{AB}).  
	
	 Our contributions are structured as follows:

	$1.$ We prove existence of solutions to a general optimization problem {\bf (GDP)} (Theorem ~\ref{thm1ND}). As particular case, we extend the existence of solutions to {\bf (DP)} from tight (or Radon) to separable Baire measures (Corollary ~\ref{Dep}).
	
	$2.$ We establish a natural  duality between {\bf (DP)} and {\bf (OP)} under general conditions (Theorem ~\ref{thm2}).
	
	$3.$ When restricted to Radon probability measures $\mu$ and $\nu$, we prove that $\Gamma(\mu, \nu)=\Pi(\mu, \nu)$ and so the problem {\bf (OP)} coincides with the classical {\bf (KP)} (Theorem ~\ref{lemma3}). This brings us back to classical duality as a special case (Corollary ~\ref{thm31}).
	\vskip5mm
\paragraph{\bf A few details about our approach and our main results} Optimal transport and the Monge-Kantorovich duality are highly active areas of research. Classical results are given on Polish spaces (a Polish space is a topological space homeomorphic to a complete separable metric space) and show, under certain conditions on the cost function $c$, that the duality $v_{\max}{\bf (DP)}= v_{\min}{\bf (KP)}$ is true. In addition, with a few more conditions, both problems also have solutions (see \cite[Theorem 1.39]{Sfilippo} and \cite[Theorem 5.10]{Villani}). An elegant proof of Kantorovich's duality in completely regular Hausdorff spaces is provided by D.A. Edwards in \cite{Eda}, where the duality has been established. However, as mentioned by the author at the end of the article, from a linear programming point of view, the discussion has been incomplete in the sense that it does not shed any light on the question of whether the problem {\bf (DP)} has a solution. We also would like to highlight the significant work by H. G. Kellerer in \cite{Khg}, where various conditions for the validity of duality have been proposed. The author also discusses specific conditions that ensure the existence of a solution to the dual problem in $\mathcal{L}^1(\mu)\times\mathcal{L}^1(\nu)$. However, Kellerer's analysis is confined to tight (or Radon) measures, which form a specific subset of the set of all Baire measures on general completely regular Hausdorff spaces. Note that in Polish spaces, the Borel $\sigma$-algebra and the Baire $\sigma$-algebra coincide, and every Borel measure is Radon (see \cite[Theorem 7.1.7]{Bvi}). 
	
    The classical approach given in \cite{Sfilippo} and \cite{Villani} begins
	by showing that the problem ${\bf (KP)}$ has a solution (thanks to the use
	of Prokhorov’s theorem) and then follows a procedure for constructing a solution
	to the problem ${\bf (DP)}$ using the c-cyclic monotonicity of the support of the solution of ${\bf (KP)}$ (see \cite[“Idea of proof of Theorem 5.10.” p. 73-74]{Villani} or the proof of \cite[Theorem 1.39]{Sfilippo}). In the classical approach, the concept of c-cyclic monotonicity is crucial for the Monge-Kantorovich duality. Our approach is different and does not use the concept of $c$-cyclic monotonicity. We first prove the existence of a solution to the
	problem ${\bf (DP)}$ by using the Arzela-Ascoli theorem for pointwize convergence, and
	the Wheeler’s result on the continuity of separable Baire measures (\cite[Proposition 4.1]{Wrf1}, see also Proposition \ref{pW2} in this article). Then, using
	the solution of ${\bf (DP)}$ and the Lagrange multiplier rule obtained recently in \cite{BaBl}, we show that ${\bf (OP)}$ (in particular ${\bf (KP)}$) also has a solution, and moreover the duality holds. In summary, the classical approach starts from the problem {\bf (KP)} to arrive at the problem {\bf (DP)} by using the concept of $c$-cyclic monotonicity, whereas in our approach we start from the problem {\bf (DP)} to arrive at the problem {\bf (OP)} (and in consequence at the problem {\bf (KP)}), by using the Lagrange multiplier rule. Note that we are mainly interested in solutions of {\bf (DP)} which belong to the space $C_b(X)\times C_b(Y)$ rather than the larger space $\mathcal{L}^1(\nu)\times \mathcal{L}^1(\nu)$. 
	
	As a consequence of Theorem ~\ref{thm1ND}, we prove in  Corollary ~\ref{thm21}, that if $X$ and $Y$ are completely regular Hausdorff spaces, then for every separable non-negative Baires measures $\mu$ on $X$ and $\nu$ on $Y$, the dual problem ${\bf (DP)}$ (under certain conditions on the cost $c$) has a solution in $C_b(X)\times C_b(Y)$, the problems ${\bf (OP)}$ has also a solution and moreover the duality $v_{\max}{\bf (DP)}=v_{\min}{\bf (OP)}$ holds. In Theorem ~\ref{lemma3} we show that when $X$ and $Y$ are perfectly normal Hausdorff spaces (in particular Polish spaces) and $\mu$, $\nu$ are Radon probability measures on $X$ and $Y$ respectively, then problems ${\bf (KP)}$ and ${\bf (OP)}$ coincide. In this way, we recover in Corollary ~\ref{thm31} the classical Kantorovich's duality. As proved in Theorem \ref{thm2}, the ${\bf (OP)}$ problem naturally appears in a duality with ${\bf (DP)}$ problem.
	
	\vskip5mm
	This article is organized as follows: In Section ~\ref{S1}, we begin by recalling the various known classes of Baire measures and provide comparisons among them. Next, we recall the three problems to be dealt with in a little more detail. Section ~\ref{S2} concerns the existence of solutions for a general optimization problem {\bf (GDP)} (Theorem ~\ref{thm1ND}). As particular case, we extend the existence of solutions to {\bf (DP)} from tight to separable Baire measures (Corollary ~\ref{Dep}). In Section ~\ref{S3} we establish duality results between {\bf (DP)} and {\bf (OP)} problems (Theorem \ref{thm2} and Corollary \ref{thm21}). In Section ~\ref{S4}, we deduce the classical Kantorovich duality as a special case. 
	
	\section{A few reminders and definitions} \label{S1}
	
	To enhance the readability of this paper, we believe it is beneficial to revisit some fundamental definitions and concepts, along with some well-established facts.
	
	Let $X$ be a completely regular Hausdorff space. We denote $C_b(X)$ the Banach space of all real-valued bounded continuous functions and $C_b(X)^*$ its topological dual. The closed convex positive cone is denoted $C_b(X)^+$ and $(C_b(X)^+)^*$ is its positive dual cone.
	
	\subsection{Different notions of measures.} Let $\mathcal{U}(X)$ be the algebra generated by all zero-sets in $X$ (a zero-set in $X$ is a set of the form $f^{-1}(0)$, where $f$ is a real-valued continuous function),  $\mathcal{B}_a(X)$ the smallest $\sigma$-algebra of subsets of $X$ generated by the zero-sets in $X$ and $\mathcal{B}(X)$ be the smallest $\sigma$-algebra generated by the open sets (or equivalently by all closed sets). Elements of $\mathcal{B}(X)$ are called Borel sets and elements of $\mathcal{B}_a(X)$ are called Baire sets. Clearly, $\mathcal{U}(X)\subseteq \mathcal{B}_a(X)\subseteq \mathcal{B}(X)$. It is well-known and easy to see that $\mathcal{B}(X)=\mathcal{B}_a(X)$, if $X$ is a perfectly normal Hausdorff space, in particular if $X$ is a metric space (see \cite[Proposition 6.3.4 \& Corollary 6.3.5]{Bvi}). 
	
	{\bf 1. (Baire and Borel measure)} A $\sigma$-additive (or countably additive) measure on the Baire (resp. the Borel) $\sigma$-algebra $\mathcal{B}_a(X)$ (resp. $\mathcal{B}(X)$) is called a Baire (resp. a Borel) measure on $X$. 
	
	{\bf 2. (Radon measure)} A Borel measure $\mu$ on $X$ is called a Radon measure if for every $B$ in
	$\mathcal{B}(X)$ and $\varepsilon >0$, there exists a compact set $K_\varepsilon \subset B$ such that $|\mu|(B\setminus K_\varepsilon) < \varepsilon$, where $|\mu|$ denotes the total variation of $\mu$.
	
	{\bf 3. (Tightness)} A non-negative set function $\mu$ defined on some system
	$\mathcal{A}$ of subsets of a topological space X (e.g., $\mathcal{A}=\mathcal{U}(X)$, $\mathcal{B}_a(X)$ or $\mathcal{B}(X)$) is called tight on $\mathcal{A}$ if for every $\varepsilon>0$,
	there exists a compact set $K_\varepsilon$ in $X$ such that $\mu(A) < \varepsilon$ for every element A
	in $\mathcal{A}$ that does not meet $K_\varepsilon$. 
	
	An additive set function $\mu$ of bounded variation on an algebra  is called tight if its total variation $|\mu|$ is tight.
	
	{\bf 4. (Regularity)} A non-negative set function $\mu$ defined on some system
	$\mathcal{A}$ of subsets of a topological space $X$ is called (inner) regular if for every $A$ in $\mathcal{A}$  and
	every $\varepsilon >0$, there exists a closed set $F_\varepsilon$ such that $F_\varepsilon\subset A$, $A\setminus F_\varepsilon \in \mathcal{A}$ and $\mu(A\setminus F_\varepsilon) < \varepsilon$. 
	
	An additive set function $\mu$ of bounded variation on an algebra is called regular if its total variation $|\mu|$  is regular.
	
	Every Radon measure on a Hausdorff space is regular and
	tight. Conversely, if a Borel measure is regular and tight, then it is Radon,
	since the intersection of a compact set and a closed set is compact. However,
	a regular Borel measure may fail to be tight (\cite[Example 7.1.6]{Bvi}). Every Baire measure is regular . Moreover, for every Baire set $E$ and every $\varepsilon >0$, there exists a continuous function $f$ on $X$ such that $f^{-1}(0)\subseteq  E$ and $|\mu |(E\setminus f^{-1}(0)) < \varepsilon$ (\cite[Corollary 7.1.8]{Bvi}).
	
	{\bf 5. (The space $\mathcal{M}(X)$ of charges)} By $\mathcal{M}(X)$ we denote the space of all additive (inner) regular set functions $m:\mathcal{U}(X) \to \R$, that is, the space of set functions that are: $(i)$ additive, $(ii)$ uniformly bounded, and $(iii)$ for any $A\in \mathcal{U}(X)$, any $\varepsilon >0$ there exists a zero-set $F$ such that $F\subseteq A$ and $|m(B)|< \varepsilon$ for all $B\subseteq A\setminus F$, $B\in \mathcal{U}(X)$.
	
	{\bf 6. (The space $\mathcal{M}_\sigma(X)$ of Baires measures)} We denote $\mathcal{M}_{\sigma}(X)$ the set of all  Baire measures, that is, the set of all finite, real-valued, regular and $\sigma$-additive measures on $\mathcal{B}_a(X)$. The set $\mathcal{M}_{\sigma}(X)$ is naturally a subspace of $\mathcal{M}(X)$.
	
	{\bf 7. (The space $\mathcal{M}_\tau(X)$ of $\tau$-additive Baires measures)} A Baire measure $\mu \in \mathcal{M}_\sigma(X)$ is said to be $\tau$-additive if, $|\mu|(Z_\alpha) \to 0$ for every net of zero sets $Z_\alpha$ decreasing to the null set.
	
	The space of all $\tau$-additive Baire measures will be denoted $\mathcal{M}_{\tau}(X)$.

	{\bf 8. (The space $\mathcal{M}_s(X)$ of separable Baire measures)} A Baire measure $\mu$ on $X$ is called separable if for every bounded continuous pseudo-metric $d$ on $X$, $\mu$ is concentrated in a subset of $X$ which has a countable set dense for $d$. This notion was introduced by R. M. Dudley in \cite{Drm} and studied by R. F. Wheeler \cite{Wrf,Wrf1} (see also the references therein). The set of all separable signed measure of $\mathcal{M}_{\sigma}(X)$ will be denoted $\mathcal{M}_s(X)$. Notice that in \cite{Wrf1}, the space $\mathcal{M}_s(X)$ is defined as the dual of $C_b(X)$ when equipped with an appropriate strict topology $\beta_e$, but it is shown in \cite[Proposition 4.1]{Wrf1}  that the topological dual of $(C_b(X),\beta_e)$ (denoted also $\mathcal{M}_s(X)$) coincides with the separable measure space introduced by M. Dudley.
	
	{\bf 9. (The space $\mathcal{M}_t(X)$ of tight Baire measures)} The subspace of $\mathcal{M}_\sigma(X)$ consisting on all tight Baire measures will be denoted by $\mathcal{M}_t(X)$. 
	
	We always have the following relationships:
	$$\mathcal{M}_t(X) \subseteq \mathcal{M}_\tau(X) \subseteq \mathcal{M}_s(X)\subseteq \mathcal{M}_\sigma(X)\subseteq \mathcal{M}(X).$$ 
	The symbols $\mathcal{M}_t^+(X)$, $\mathcal{M}_\tau^+(X)$, $\mathcal{M}_s^+(X)$, $\mathcal{M}_\sigma^+(X)$ (resp. $\mathcal{P}_t(X)$, $\mathcal{P}_\tau(X)$, $\mathcal{P}_s(X)$, $\mathcal{P}_\sigma(X)$) stand, respectively, for the corresponding classes of non-negative measures (resp. probability measures). The set $\mathcal{M}^+(X)\subseteq \mathcal{M}(X)$ denotes the set of all non-negative additive regular set function and $\mathcal{M}^1(X)$ the subset of $\gamma \in \mathcal{M}^+(X)$ such that $\gamma(X)=1$.
	
	In general, these inclusions are strict. However, if $X$ is paracompact, in particular if $X$ is metrizable, then $\mathcal{M}_\tau(X)= \mathcal{M}_s(X)$ (see \cite[Proposition 3.7 \& Proposition 3.8]{Wrf1}). If $X$ is a complete metric space, then $\mathcal{M}_t(X)= \mathcal{M}_s(X)$ (see \cite[Corollary p. 257]{Drm}). If $X$ is a separable metric space, $\mathcal{M}_\tau(X)= \mathcal{M}_\sigma(X)$ \cite[Proposition 7.2.2]{Bvi}. We have $\mathcal{M}_s(X)= \mathcal{M}_\sigma(X)=\mathcal{M}(X)$ if and only if $X$ is pseudocompact (see \cite[Theorem 2.3 \& Theorem 2.6]{Wrf}). 
	
	Anyway, for a Polish space $X$ we have, by the preceding comments, that $\mathcal{B}_a(X)=\mathcal{B}(X)$ and $\mathcal{M}_t(X)=\mathcal{M}_\tau(X)=\mathcal{M}_s(X)=\mathcal{M}_\sigma(X)$. However, we can have $\mathcal{M}_t(X)\subsetneq \mathcal{M}_s(X)=\mathcal{M}_\sigma(X)$ even in a non-complete separable metric space. Indeed, there exists a non-Lebesgue measurable subset $M$ of $[0,1]$ and a regular Borel (separable) probability measure $\mu $ on the non-complete separable metric space  $(M,|\cdot|)$, which is not Radon (see \cite[Example 7.1.6 $\&$ Example 1.12.13]{Bvi}). For more details, we refer to \cite{AB, Bvi, Drm, Va, Wrf1, Wrf}. 
	
	\subsection{Monge-Kantorovich problems} 
	
	Let $Z$ be a completely regular Hausdorff space. By the Alexandroff representation theorem \cite[Theorem 7.9.1]{Bvi}, for any bounded positive linear functional $L$ on $C_b(Z)$ ($L\in (C_b(Z)^+)^*$) there exists a non-negative additive regular
	set function $\gamma$ on $\mathcal{U}(Z)$ with $\gamma(Z)=\|L\|$ such that $L (f)=\int_Z f d\gamma$ for all $f\in C_b(Z)$. The set $\mathcal{M}^1(Z)$ of all non-negative additive regular set function $\gamma$ such that $\gamma(Z)=1$, can be identified with the convex weak$^*$-compact subset of $(C_b(Z)^+)^*$ (Banach-Alaoglu theorem's) defined by
	$$D^1(Z):=\overline{\textnormal{conv}}^{w^*}\left\{\delta_{z}: z\in Z \right\},$$
	where, for each $z\in Z$, $\delta_{z}: C_b(Z)\to \R$ denotes the Dirac evaluation at $z$,  $\delta_{z}(f)=f(z)$ for all $f\in C_b(Z)$. By definition, a net $(\gamma_\alpha)_\alpha$ in $\mathcal{M}^1(Z)$, weak-converges to $\gamma$ if and only if $\int_Z f d\gamma_\alpha \to \int_Z f d\gamma$, for every $f\in C_b(Z)$. Notice that the terminology of weak-convergence of additive set function or measures coincides with the weak$^*$-convergence of the associated linear continous functionals in $C_b(Z)^*$. As an immediate consequence of Banach-Alaoglu theorem's (up to the Alexandroff representation theorem) we have the following proposition.
	
	\begin{proposition} \label{propweak} 
		Let $Z$ be a completely regular Hausdorff space. Then, $(\mathcal{M}^1(Z),w)$ is convex and  weak-compact subset of $(\mathcal{M}(Z), w)$.
	\end{proposition} 
	
	We denote $lsc_{bi}(Z)$, the set of all proper bounded from below lower semicontinuous functions $g:Z\to \R\cup \{+\infty\}$. In a normal Hausdorff space $Z$, each proper bounded from below lower semicontinuous function $g$ satisfies $g=\sup\{f: f\in C_b(Z), f\leq g\}$ (see \cite[Theorem 2, Corollary 1 \& Example 2]{Ba1}, see also \cite{Th}).  Every positive linear continuous functional $L\in (C_b(Z)^+)^*$ can be extended to a lower semicontinuous convex function from $lsc_{bi}(Z)$ to $\R\cup\{+\infty \}$ as follows: $$ L_*(g):=\sup\{L(f): f\in C_b(Z): f\leq g\}, \hspace{1mm} \forall g\in lsc_{bi}(Z).$$
	We will use the notation $L(g)$, to designate  $L_*(g)$ for every $g \in lsc_{bi}(Z)$.  In this way, for every $\gamma \in \mathcal{M}^1(Z)$ and every $g\in lsc_{bi}(Z)$, let us define the lower integral of $g$ by $$I_\gamma(g):=\sup \left \{\int_Z f d\gamma: f\in C_b(Z): f\leq g \right \}.$$
As stated in the following known proposition, if $\gamma \in \mathcal{M}^+_\sigma(Z)$ is a positive $\sigma$-additive measure then the lower integral of $g\in lsc_{bi}(Z)$ coincides with the usual integrale of $g$. For this reason, we will use the notation $\int_Z g d\mu$ instead of $I_\gamma(g)$, as soon as $\gamma \in \mathcal{M}^1(Z)$ and $g\in lsc_{bi}(Z)$.
	\begin{proposition} Let $Z$ be a perfectly normal Hausdorff space and $\gamma \in \mathcal{M}^+_\sigma(Z)$. Then, $I_\gamma(g)=\int_Z g d\gamma$, for every $g\in lsc_{bi}(Z)$.
		
	\end{proposition}
	\begin{proof} Let $g\in lsc_{bi}(Z)$. We can assume without loss of generality that $g\geq 0$. Since $Z$ is perfectly normal Hausdorff space, $g$ is the limit of a point-wise increasing sequence of continuous functions (see \cite{Th}). So, there exists a sequence $(f_n)$ of bounded continuous functions with
		$$0\leq f_0(x)\leq f_1(x)\leq ... \leq f_n(x)\leq ...\leq g(x) \textnormal{ and } g(x)=\lim_{n\to +\infty} f_n(x), \hspace{1mm} \forall x\in Z.$$
	By the monotone convergence theorem, $\lim_{n\to +\infty} \int_Z f_n d\gamma=\int_Z \lim_{n\to +\infty} f_n d\gamma=\int_Z g d\gamma$. On the other hand, for all $n\in \N$, $\int_Z f_n d\gamma \leq I_\gamma(g)$. It follows that $\int_Z g d\gamma \leq I_\gamma(g)$. Now, we have $\int_Z f d\gamma \leq \int_Z g d\gamma$, for every $f\leq g$, $f\in C_b(Z)$. Thus, $I_\gamma(g)\leq \int_Z g d\gamma$. 
	\end{proof}
We recall the three  optimisation problems that we are going to treat.
	
	\vskip5mm 
	
	\paragraph{\bf A. The problem {\bf (OP)}}  	Let $X$ and $Y$  be two completely regular Hausdorff spaces, $\mu \in \mathcal{P}_\sigma(X)$ and $\nu \in \mathcal{P}_\sigma(Y)$. We introduce the optimization problem {\bf (OP)} as follows:
	\begin{eqnarray*} \label{OP}
		{\bf (OP)} \hspace{5mm} v_{\min}{\bf (OP)}=\inf\left \{\int_{X\times Y} c d\gamma: \gamma \in \Gamma(\mu,\nu)\right \},
	\end{eqnarray*}
	where, 
	\begin{eqnarray*} \Gamma(\mu,\nu):= \biggl\{\gamma \in  \mathcal{M}^1(X\times Y):&& \int_{X\times Y} \phi\oplus \psi d\gamma =\int_X \phi d\mu +\int_Y \psi d\nu; \\
	&&	\phi\in C_b(X), \psi \in C_b(Y)\biggr\}.
	\end{eqnarray*}
	Notice that the usual product measure $\mu\otimes \nu$ defined on the product $\sigma$-algebra $\mathcal{B}_a(X)\otimes \mathcal{B}_a(Y)$ can be extended to an element $\gamma \in \mathcal{M}^1(X\times Y)$. Indeed, the linear continuous functional $L(f\oplus g):=\int_X fd\mu+ \int_Y g d\nu$, defined on the subspace $C_b(X)\oplus C_b(Y)$ of $C_b(X\times Y)$ can be extended to a continuous linear functional $\widetilde{L}$ on $C_b(X\times Y)$ with $\widetilde{L}(1_{X\times Y})=1$, thanks to the Hanh-Banach theorem. By the Alexandroff representation theorem \cite[Theorem 7.9.1]{Bvi}, the functional $\widetilde{L}$ can be represented by an element $\gamma \in \mathcal{M}^1(X\times Y)$. Thus, $\Gamma(\mu, \nu)$ is always nonempty. Moreover, $\Gamma(\mu,\nu)$ is convexe and weak-closed in $\mathcal{M}^1(X\times Y)$, so it is convex weak-compact set (see Proposition \ref{propweak}).
	\begin{remark}
		Suppose that $X$ and $Y$ are completely regular Hausdorff pseudocompact spaces such that $X\times Y$ is pseudocompact (see \cite{St}). It is known that in this case, the space $\mathcal{M}(X\times Y)$  of  finite, finitely-additive, zero-set regular set functions on $X\times Y$ coincides with the space  $\mathcal{M}_s(X\times Y)$ of separable Baire measures on $X\times Y$ (see \cite[Theorem 2.3 \& Theorem 2.6]{Wrf}). In particular, we have $\mathcal{M}^1(X\times Y)=\mathcal{P}_s(X\times Y)$ and so every $\gamma \in \Gamma(\mu,\nu)$ is a separable Baire probability measure on $X\times Y$, whenever $\mu\in \mathcal{P}_s(X)=\mathcal{P}_\sigma(X)$ and $\nu \in \mathcal{P}_s(Y)=\mathcal{P}_\sigma(Y)$.
	\end{remark}
	\vskip5mm
	\paragraph{\bf B. The Kantorovich primal problem}  Given a Borel measure $\mu$ on $X$ and measurable map $f: X\to Z$ (a topological space), the pushforward measure of $\mu$ by $f$ is denoted $f_{\#} \mu$ and defined by 
	$$\forall B\in \mathcal{B}(Z)=\mathcal{B}(Z): f_{\#} \mu(B):=\mu(f^{-1}(B)).$$
	
	We denote $\mathcal{P}_r(X)$ the set of all Radon probability measures on $X$. 	Let $X$ and $Y$  be two completely regular Hausdorff spaces, $\mu \in \mathcal{P}_r(X)$ and $\nu \in \mathcal{P}_r(Y)$ be two Radon probability measures. A Radon probability measure defined on $\mathcal{B}(X\times Y)$, say $\gamma \in \mathcal{P}_r(X\times Y)$, is called a transport plan between $\mu$ and $\nu$ if it satisfies both $(\pi_X)_{\#} \gamma=\mu$ and $(\pi_Y)_{\#} \gamma=\nu$, where $\pi_X(x,y)=x$ and $\pi_Y(x,y)=y$ are the two projections of $X\times Y$ onto $X$ and $Y$, respectively. The set of all transport plans, is
	$$\Pi(\mu,\nu):=\{\gamma \in \mathcal{P}_r (X\times Y): (\pi_X)_{\#} \gamma=\mu, (\pi_Y)_{\#} \gamma=\nu\}.$$
	The classical Kantorovich primal problem is then given by the following minimization problem:
	\begin{eqnarray*} \label{KP}
		{\bf (KP)} \hspace{5mm} v_{\min}{\bf (KP)}:=\inf\left \{\int_{X\times Y} c d\gamma: \gamma \in \Pi(\mu,\nu)\right \}.
	\end{eqnarray*} 
	
	\vskip5mm
	\paragraph{\bf C. The  Kantorovich dual problem {\bf (DP)}} Let $X$ and $Y$  be two completely regular Hausdorff spaces, $\mu \in \mathcal{M}^+_\sigma(X)$ and $\nu \in \mathcal{M}^+_\sigma(Y)$. The  Kantorovich dual problem is given by the following maximization problem:
	\begin{eqnarray*} \label{DP}
		{\bf (DP)}	\hspace{5mm} v_{\max}{\bf (DP)}:=\sup \left \{ \int_X \phi d\mu+\int_Y \psi d\nu: \phi\in C_b(X), \psi\in C_b(Y); \phi\oplus \psi\leq c \right \},
	\end{eqnarray*}
	Notice that the inequality $v_{\max}{\bf (DP)}\leq v_{\min}{\bf (OP)}$ is always true for every $\mu \in \mathcal{P}_\sigma(X)$ and $\nu \in \mathcal{P}_\sigma(Y)$ and
	$$v_{\max}{\bf (DP)}\leq v_{\min}{\bf (OP)} \leq v_{\min}{\bf (KP)},$$ 
	is true, whenever $\mu \in \mathcal{P}_r(X)$ and $\nu \in \mathcal{P}_r(Y)$. 
	
	
	\section{Existence of solutions for {\bf (DP)}}\label{S2}
	
		In this section, we will discuss the existence of solutions for problem  ${\bf(DP)}$ under certain conditions on $c$. To do this, we need to introduce a few definitions, notions and intermediate results. Let $c: X\times Y\to \R \cup\{+\infty \}$ be a proper function. 
	For each real-valued function $\xi: X\to \R$, we define the $c$-transform $\xi^c: Y\to [-\infty,+\infty)$ by $$ \xi^c(y):=\inf_{x\in X}\{c(x,y)-\xi(x)\}, \hspace{1mm} \forall y\in Y.$$ 
	Similarily, for each $\zeta: Y \to \R$, we define the $\bar c$-transform $\zeta^{\bar c}: X\to [-\infty, +\infty)$ by 
	$$\zeta^{\bar c}(x):=\inf_{y\in Y}\{c(x,y)-\zeta(y)\}, \hspace{1mm} \forall x\in X.$$
	Notice that $ \xi^c$ and $\zeta^{\bar c}$ are bounded whenever $c$, $\xi$ and $\zeta$ are bounded. Set, 
	$$ \underline{d}_c(y,y'):=\sup_{x\in X}|c(x,y)-c(x,y')|,\hspace{1mm} \forall y, y'\in Y,$$
	$$ \overline{d}_c(x,x'):=\sup_{y\in Y}|c(x,y)-c(x',y)|,\hspace{1mm} \forall x, x'\in X.$$
	\begin{definition}
		We say that $c$ has the property $(H)$ if the following assertion holds: For each $x_0\in X$ and $y_0\in Y$: $\lim_{x\to x_0} \overline{d}_c(x,x_0)=0$ and $\lim_{y\to y_0} \underline{d}_c(y,y_0)=0.$ 
	\end{definition}
	
	The property $(H)$ means that $c$ is equicontinuous with respect to each variable, that is, the family $\{c(\cdot,y): y\in Y\}\subseteq C_b(X)$ and $\{c(x,\cdot): x\in X\}\subseteq C_b(Y)$ are equicontinuous at each point of $X$ and $Y$ respectively. In particular, $c$ is jointly continuous at each point. Every real-valued uniformly continuous function on uniform space $X\times Y$, satisfies the property $(H)$. However, a continuous function can possess the property $(H)$ without being uniformly continuous, for example if $c(x,y)=a(x)+b(y)+\ell(x,y)$ for all $x\in X$, $y\in Y$, where $a\in C_b(X)$, $b\in C_b(Y)$ are continuous not uniformly continuous, and $\ell\in C_b(X\times Y)$ is uniformly continuous. We can see from \cite[Theorem 2.2]{Fz} that if $X$ and $Y$ are completely regular infinite spaces, then every bounded continuous function $c$ has the property $(H)$ if and only if $X\times Y$ is a pseudocompact space.

	The following proposition is easy to establish. It is therefore left without proof.
	\begin{proposition} \label{Lip} Assume that $c: X\times Y\to \R$ is a bounded function. Then, for each bounded function $\phi: X\to \R$, we have
		\begin{eqnarray} \label{eq1} \forall y,y'\in Y; |\phi^c(y)-\phi^c(y')|\leq \underline{d}_c(y,y'),
		\end{eqnarray} 
		\begin{eqnarray} \label{eq2}
			\forall x,x'\in X; |\phi^{c\bar c} (x)-\phi^{c\bar c}(x')|\leq \overline{d}_c(x,x'),
		\end{eqnarray} 
		where $\phi^{c\bar c}:=(\phi^{c})^{\bar c}$. In other words, $\phi^{c\bar c}$ and  $\phi^c$ are $1$-Lipschitz functions with respect to the pseudometrics $\overline{d}_c$ and $\underline{d}_c$ on $X$ and $Y$ respectively.
	\end{proposition}

	\begin{remark} If we assume that $(X,d_X)$ and $(Y,d_Y)$ are metric spaces and $c$ is a Lipschitz function with respect to the metric $d_{X\times Y}((x,x'),(y,y')):=d_X(x,x')+d_Y(y,y')$, then $\phi^{c\bar c}$ and  $\phi^c$ are Lipschitz functions on $X$ and $Y$ respectively. 
		
	\end{remark}

	The topology on $C_b(X)$ of pointwise convergence on $X$ will be denoted $\mathcal{T}_p$, (or $\mathcal{T}_p(X)$ if confusion might arise). Let $\mathcal{E}(X)$  denote the family of uniformly bounded subsets of $C_b(X)$ which are equicontinuous.
	
	\begin{proposition} \label{phi}  Assume that $c: X\times Y\to \R$ is a bounded function with the property $(H)$ and set  
		$$A:=\{\phi^{c\bar c} +\inf_{Y}\phi^c: \phi\in C_b(X)\}\subseteq C_b(X),$$
		$$B:=\{\phi^c -\inf_{Y}\phi^c: \phi\in C_b(X)\}\subseteq C_b(Y).$$
		Then, $\overline{A}^{\mathcal{T}_p} \in \mathcal{E}(X)$ and $\overline{B}^{\mathcal{T}_p} \in \mathcal{E}(Y)$. In particular, $\overline{A}^{\mathcal{T}_p}$ and $\overline{B}^{\mathcal{T}_p} $ are compact sets in $(C_b(X),\mathcal{T}_p)$ and $(C_b(Y), \mathcal{T}_p)$ respectively and in consequence, $\overline{A}^{\mathcal{T}_p}\times \overline{B}^{\mathcal{T}_p}$ is a compact set in the product topological space $(C_b(X),\mathcal{T}_p)\times (C_b(Y),\mathcal{T}_p)$.
		
	\end{proposition}
	\begin{proof}  From the property $(H)$ of $c$ and using $(\ref{eq1})$, we have that the family $\{\phi^c: \phi\in C_b(X)\}$ is an equicontinuous subset of $C_b(Y)$. Similarly, using $(\ref{eq2})$, the family $\{\phi^{c \bar c}: \phi\in C_b(X)\}$ is an equicontinuous subset of $C_b(X)$. Now, we prove that $A$ and $B$ are uniformly bounded. Since $\phi^c$ is bounded on $Y$, we have $\phi^c -\inf_{Y}\phi^c\geq 0$. On the other hand, using $(\ref{eq1})$, we see that for all $y\in Y$
		\begin{eqnarray*}
			\phi^c(y) - \inf_{Y}\phi^c\leq \sup_{y'\in Y} \sup_{x\in X}|c(x,y)-c(x,y')|\leq 2\|c\|_{\infty}.
		\end{eqnarray*}
		Thus, $\phi^c -\inf_{Y}\phi^c \in [0, 2\|c\|_{\infty}]$, for all $\phi \in C_b(X)$. Hence, $B$ is uniformly bounded in $C_b(Y)$. From the definitions, we have  $\phi^{c\bar c}: X\to \R$ and 
		\begin{eqnarray*}
			\phi^{c\bar c}(x) &=&\inf_{y\in Y}\{c(x,y)-\phi^c(y)\}.
		\end{eqnarray*}
		Thus, 
		\begin{eqnarray*}
			\phi^{c\bar c}(x) + \inf_{y\in Y} \phi^{c}&=&\inf_{y\in Y}\{c(x,y)-(\phi^c(y)-\inf_{y\in Y} \phi^{c})\}.
		\end{eqnarray*}
		Since $\phi^c-\inf_{y\in Y} \phi^{c} \in [0,2\|c\|_{\infty}]$, then for all $x\in X$ and $y\in Y$ we have $$-3\|c\|_{\infty}\leq c(x,y) -2\|c\|_{\infty} \leq c(x,y)-(\phi^c(y)-\inf_{y\in Y} \phi^{c}) \leq c(x,y)\leq \|c\|_{\infty}.$$
		This gives that $\phi^{c\bar c}(x) + \inf_{y\in Y} \phi^{c}\in [-3\|c\|_{\infty}, \|c\|_{\infty}]$. Hence, $A$ is uniformly bounded in $C_b(X)$. Finally, we proved that $A$ (resp. $B$) is uniformly bounded and equicontinuous in $C_b(X)$ (resp. in $C_b(Y)$). Their $\mathcal{T}$-closure  $\overline{A}^{\mathcal{T}_p}$ and $\overline{B}^{\mathcal{T}_p}$ remains uniformly bounded. By the Arzela-Ascoli theorem \cite[Theorem T2, XX, 2; 1]{Sl}, they are also equicontinuous in $(C_b(X),\mathcal{T}_p)$ and $(C_b(Y),\mathcal{T}_p)$ respectively. They are in particular $\mathcal{T}_p$-compact \cite[Theorem T2, XX, 4; 1]{Sl}.
		
	\end{proof}
	
	Now we are going to give the main result of this section. Let $R: C_b(X)\times C_b(Y)\to \R \cup \{-\infty\}$ be a proper function. We say that $R$ is $\oplus$-nondecreasing if for all $(\phi,\psi); (\phi' , \psi') \in C_b(X)\times C_b(Y)$: 
	$$\phi\oplus \psi \leq \phi' \oplus \psi'  \Longrightarrow R(\phi, \psi)\leq R(\phi', \psi'),$$
	where $\phi\oplus \psi \leq \phi' \oplus \psi'$ if and only if $\phi(x)+\psi(y)\leq \phi'(x)+\psi'(y)$ for every $(x,y)\in X\times Y$. Consider the following maximization problem:
	\begin{eqnarray} \label{GDP}
		{\bf (GDP)}	\hspace{5mm} v_{\max}{\bf (GDP)}:=\sup \left \{ R(\xi,\zeta):  \xi\oplus \zeta\leq c; \hspace{1mm} \xi\in C_b(X), \zeta\in C_b(Y)\right \}.
	\end{eqnarray}
	
	\begin{theorem} \label{thm1ND} Let $X$ and $Y$ be two completely regular Hausdorff spaces and $c: X\times Y\to \R$ be a bounded function satisfying the property $(H)$. Let $R: C_b(X)\times C_b(Y) \to \R \cup \{-\infty\}$ be a proper function. Suppose that:
		
		$(i)$ $R$ is $\oplus$-nondecreasing. 
		
		$(ii)$ $R$ is upper semicontinuous on the product space $(C, \mathcal{T}_p) \times (D,\mathcal{T}_p)$ for each $C\in \mathcal{E}(X)$ and $D\in \mathcal{E}(Y)$.
		
		Then, the problem {\bf (GDP)} has a solution of the form $(\xi^{c \bar c}_0,\xi^c_0)\in C_b(X)\times C_b(Y)$ where $\xi_0\in C_b(X)$.
	\end{theorem}
	\begin{proof} It is easily seen that the admissible set of the problem ${\bf (GDP)}$ namely $$F:=\{(\xi,\zeta)\in C_b(X)\times C_b(Y): \xi(x)+\zeta(y)\leq c(x,y); \forall (x,y)\in X\times Y\}$$ is closed in $(C_b(X),\mathcal{T}_p)\times (C_b(Y),\mathcal{T}_p)$ (for the product topology of pointwize convergence). Hence, using Proposition \ref{phi}, we have that $F\cap \left(\overline{A}^{\mathcal{T}_p}\times \overline{B}^{\mathcal{T}_p}\right)$ is a compact set in $(C_b(X),\mathcal{T}_p)\times (C_b(Y),\mathcal{T}_p)$ (where, $A$ and $B$ are the sets defined in Proposition ~\ref{phi}). Since  $R$ is upper semicontinuous on  $\overline{A}^{\mathcal{T}_p}\times \overline{B}^{\mathcal{T}_p}$ by assumption, it follows by  Weierstrass theorem's, that there exists $(\xi_0,\zeta_0) \in F\cap(\overline{A}^{\mathcal{T}_p}\times \overline{B}^{\mathcal{T}_p})$ such that
		\begin{eqnarray*} \label{eqL} \max\left \{ R(\xi, \zeta) :  (\xi,\zeta)\in F\cap\left(\overline{A}^{\mathcal{T}_p}\times \overline{B}^{\mathcal{T}_p}\right)\right \}=R(\xi_0,\zeta_0).
		\end{eqnarray*}
		It remains to show that $$v_{\max}{\bf (GDP)}=\max\left \{ R(\xi,\zeta):  (\xi,\zeta)\in F\cap \left(\overline{A}^{\mathcal{T}_p}\times \overline{B}^{\mathcal{T}_p}\right)\right\}.$$
		Indeed,  it is easy to check that for every $\phi\in  C_b(X)$, we have $\phi(x)+\phi^c(y)\leq c(x,y)$ for all $(x,y)\in X\times Y$. Moreover, for each $(\phi, \psi) \in C_b(X)\times C_b(Y)$, if $\phi\oplus \psi\leq c$ then $\psi\leq \phi^c$ on $Y$. Thus, $\phi\oplus \psi\leq \phi\oplus \phi^c\leq c$. The same reasoning also gives
		\begin{eqnarray*}
			\phi(x)+\psi(y)\leq \phi(x)+\phi^c(y) &\leq& \phi^{c \bar c}(x)+\phi^c(y)\\
			& =&\left(\phi^{c \bar c}(x)+\inf_{Y}\phi^c\right) +\left(\phi^c(y)-\inf_{Y}\phi^c\right)\\
			&\leq& c(x,y).
		\end{eqnarray*}
		Recall that $\phi^{c \bar c}(\cdot)+\inf_{Y}\phi^c \in A$ and $\phi^c(\cdot)-\inf_{Y}\phi^c \in B$. In other words, for each $\phi\in  C_b(X)$, $\psi\in  C_b(Y)$ such that $\phi\oplus \psi\leq c$, there exists $(\xi,\zeta)\in  F\cap(\overline{B}^{\mathcal{T}_p}\times \overline{A}^{\mathcal{T}_p})$ such that $\phi\oplus\psi\leq \xi\oplus\zeta\leq c$. Since $R$ is $\oplus$-nondecreasing, it follows that for every $(\phi, \psi)\in F$, there exists $(\xi, \zeta)\in F\cap(\overline{B}^{\mathcal{T}_p}\times \overline{A}^{\mathcal{T}_p})$, such that $R(\phi, \psi)\leq R(\xi, \zeta)$. This gives 
		\begin{eqnarray*} \label{KP}
			v_{\max}{\bf (GDP)} &=&\max\left \{ R(\xi, \zeta):  (\xi,\zeta)\in F\cap\left (\overline{B}^{\mathcal{T}_p}\times \overline{A}^{\mathcal{T}_p}\right )\right \}\\
			&=&R(\xi_0, \zeta_0).
		\end{eqnarray*}
		Using the same reasoning as above, we have for all $x\in X$, $y\in Y$
		\begin{eqnarray*}
			\xi_0(x)+\zeta_0(y)\leq \xi_0(x)+\xi^c_0(y) &\leq& \xi^{c \bar c}_0(x)+\xi^c_0(y)\\
			&\leq& c(x,y).
		\end{eqnarray*}
		Using again the fact that $R$ is $\oplus$-nondecreasing, we get  $v_{\max}{\bf (GDP)}= R(\xi^{c \bar c}_0 ,\xi^c_0)$, with $\xi_0\in C_b(X)$.
		
	\end{proof}
	We are going to give an important consequence of the previous theorem, namely, the existence of solutions for the problem {\bf (DP)}. We need the following fundamental proposition due to Wheeler \cite{Wrf1}.
		\begin{proposition} \label{pW2} $($\textnormal{R. F. Wheeler} \cite[Proposition 4.1]{Wrf1}$)$ Let $X$ be a completely regular Hausdorff space and $\mu \in \mathcal{M}_\sigma(X)$. Then, $\mu \in \mathcal{M}_s(X)$ if and only if the restriction of $\mu$ to each member of $\mathcal{E}(X)$ is $\mathcal{T}_p$-continuous.
	\end{proposition}
	
		\begin{corollary} \label{Dep} Let $X$ and $Y$ be two completely regular Hausdorff spaces and $c: X\times Y\to \R$ be a bounded function satisfying the property $(H)$. Let $I$ be an arbitrary nonempty set and for each $\alpha \in I$, $\mu_\alpha \in \mathcal{M}_s^+ (X)$ and $\nu_\alpha\in \mathcal{M}_s^+ (Y)$ be non-negative separable Baire measures. Define, $R: C_b(X)\oplus C_b(Y) \to \R \cup \{-\infty\}$ by 
		$$R(\xi, \zeta)=\inf_{\alpha \in I} \left(\int_X \xi d \mu_\alpha +\int_Y \zeta d\nu_\alpha\right), \hspace{1mm} \forall (\xi, \zeta)\in C_b(X)\times C_b(Y).$$
		Then, the problem {\bf (GDP)} (see $(\ref{Dep})$) has a solution of the form $(\xi^{c \bar c}_0,\xi^c_0)\in C_b(X)\times C_b(Y)$ where $\xi_0\in C_b(X)$. In particular (with $I=\{1\}$), the problem {\bf (DP)} has a solution.
	\end{corollary}
	\begin{proof}
		First, let us observe that $R$ is a proper function, since $R(1_{X},1_{Y})=2$. It is easy to see that, for each $\alpha \in I$, the functional $(\xi, \zeta)\to \int \xi d\mu_\alpha +\int \zeta d\nu_\alpha$ is $\oplus$-nondecreasing since $\mu_\alpha$, $\nu_\alpha$ are non-negative measures. It follows that $R$ is $\oplus$-nondecreasing. By Proposition \ref{pW2}, for each $\alpha \in I$, the functional $(\xi, \zeta)\to \int_X \xi d\mu_\alpha +\int_Y \zeta d\nu_\alpha$, defined on the product space $(C_b(X),\mathcal{T}_p)\times (C_b(Y),\mathcal{T}_p)\to \R$ is continuous when restricted to $C\times D$, for every $C\in \mathcal{E}(X)$ and every $D\in \mathcal{E}(X)$. Thus, as an infimum of continuous functions, $R$ is upper semicontinuous on the product space $(C, \mathcal{T}_p) \times (D,\mathcal{T}_p)$. We can now apply Theorem \ref{thm1ND} to conclude.
	\end{proof}

	\section{Duality between {\bf (DP)} and {\bf (OP)}} \label{S3}
	
		This section is dedicated to our main result on the duality between problems {\bf (DP)} and {\bf (OP)}. Before giving the main result Theorem \ref{thm2}, we need some preliminary results. 
		
		Let us begin by noting that the problem {\bf (OP)} always has solutions. This can be deduced in a fairly elementary way from the following proposition.
		
		\begin{proposition} \label{weaksci} Let $Z$ be a completely regular Haudorff space. For each $g\in lsc_{bi}(Z)$, the function $\widehat{g}: \left (\mathcal{M}^1(Z), w\right)\to \R\cup\{+\infty\}$ define by $\widehat{g}(\gamma):=\int_Z g d\gamma$ for all $\gamma \in \mathcal{M}^1(Z)$, is proper weak-lower semicontinuous. Equivalently, for each $g\in lsc_{bi}(Z)$, the function $\widehat{g}: \left (D^1(Z), w^*\right)\to \R\cup\{+\infty\}$ define by $$\widehat{g}(L):=L(g):=\sup_{f\leq g, f\in C_b(Z)} L(f), \hspace{1mm} \forall L\in D^1(Z)\subseteq (C_b(Z))^*,$$  is proper weak$^*$-lower semicontinuous.
		\end{proposition}
		
		\begin{proof}
			The function $\widehat{g}$ is proper. Indeed, there exists some $z_0\in Z$ such that $g(z_0)<+\infty$. Thus, $\widehat{g}(\widehat{\delta}_{z_0})=g(z_0)<+\infty$, where $\widehat{\delta}_{z_0}$ is the Dirac measure defined by $\widehat{\delta}_{z_0}(A)=1$ if $z_0\in A$ and $\widehat{\delta}_{z_0}(A)=0$ if $z_0\not \in A$, for every Borel set $A\in \mathcal{B}(Z)$. Now, $\widehat{g}$ is weak-lower semicontinuous as a supremum of weak-continuous functions, since for each $\gamma \in \mathcal{M}^1(Z)$,  $\widehat{g}(\gamma)=\int_Z g d\gamma:=\sup\{\int_Z f d\gamma : f\in C_b(Z), f\leq g\}$ and $\gamma \mapsto \int_Z f d\gamma$ is weak-continuous for each $f\in C_b(Z)$.
		\end{proof}

		\begin{proposition} \label{propOPS} Let $X$ and $Y$ be two completely regular Hausdorff spaces, $\mu \in \mathcal{P}_\sigma(X)$ and $\nu \in \mathcal{P}_\sigma(Y)$ be two probability Baire measures, and $c\in lsc_{bi}(X\times Y)$ be a function such that $v_{\min}{\bf (OP)}<+\infty$. Then the set $\Gamma(\mu, \nu)$ is (convex) weak-compact Hausdorff space and the problem {\bf (OP)} has a solution. 
		\end{proposition}
		\begin{proof} 
			Clearly, the set $\Gamma(\mu, \nu)$ is (convex) weak-closed in the weak-compact Hausdorff space $(\mathcal{M}^1(X\times Y), w)$ (Proposition \ref{propweak}). It is then (convex) weak-compact. On the other hand, by Proposition \ref{weaksci}, the function $\widehat{c}$ is weak-lower semicontinuous. Hence its minimum is atained on $\Gamma(\mu, \nu)$.
		\end{proof}
			
		In order to prove Theorem \ref{thm2} (see below), we need to use tools developed recently in \cite{BaBl}. To do this, we need the following notations and concepts. We denote $\mathcal{F}:=\left\{\phi\oplus \psi: (\phi, \psi)\in C_b(X)\times C_b(Y)\right\}$ which is a normed  subspace of $(C_b(X\times Y), \|\cdot\|_{\infty})$. Let $c\in lsc_{bi}(X\times Y)$ be a fixed  function. For each $\gamma \in \mathcal{M}^1(X\times Y)$, define the linear continuous functional 
	\begin{eqnarray*}
		\hat{\gamma}: C_b(X\times Y) &\to & \R\\
		f &\mapsto & \int_{X\times Y} f d\gamma
	\end{eqnarray*}
	and the affine continuous map on $\mathcal{F}$ by 
	\begin{eqnarray*}
		\widehat{\gamma}_c : \mathcal{F} &\to & \R\\
		f &\mapsto & \int_{X\times Y} c d\gamma -\int_{X\times Y} f d\gamma =\int_{X\times Y} c d\gamma - \hat{\gamma}(f).
	\end{eqnarray*}
	The Fr\'echet-differential of $\widehat{\gamma}_c$ on $\mathcal{F}$ at any point $f\in \mathcal{F}$ is the linear continuous functional $d_F \widehat{\gamma}_c(f)=-\widehat{\gamma}_{|\mathcal{F}}: h\mapsto \int_{X\times Y} h d\gamma$, for all $h\in \mathcal{F}$ ($\widehat{\gamma}_{|\mathcal{F}}$ denotes the restriction of $\widehat{\gamma}$ to $\mathcal{F}$). Define 
	$$[\mathcal{M}^1(X\times Y)]^\times:=\left\{f\in \mathcal{F} : \widehat{\gamma}_c(f)\geq 0, \forall \gamma\in \mathcal{M}^1(X\times Y)\right\}.$$ 
	Up to the identification between $D^1(X\times Y)$ and $\mathcal{M}^1(X\times Y)$, we can easily see that
	\begin{eqnarray} \label{okb}
		[\mathcal{M}^1(X\times Y)]^\times = \left\{f=\phi\oplus \psi\in \mathcal{F}: c(x,y)-\phi\oplus \psi(x,y)\geq 0, \forall (x,y)\in X\times Y \right\}.
	\end{eqnarray}
	Given $\mu \in \mathcal{P}_\sigma(X)$ and $\nu \in \mathcal{P}_\sigma(Y)$, we define the linear continuous functional $H: \mathcal{F} \to \R$ by $H(f)=\int_X \phi d\mu +\int_Y \psi d\nu$, for every $f=\phi\oplus \psi \in \mathcal{F}$. Notice that $H$ is well defined since $\phi\oplus \psi=\phi'\oplus \psi'$ implies $\phi =\phi'+a$ and $\psi=\psi'-a$ for some $a\in \R$, which gives $\int_X \phi d\mu +\int_Y \psi d\nu= \int_X \phi' d\mu +\int_Y \psi' d\nu$. Since $H$ is a linear continuous functional, the Fr\'echet-differential of $H$ at any $f\in \mathcal{F}$ is $d_F H(f)=H$. 
	
	By using the remark in $(\ref{okb})$, we see that the problem {\bf (DP)} is equivalent to the following optimisation problem with an infinite number of inequality constraints, 
	
	\begin{equation*}
		{\bf (\widetilde{DP})}
		\left \{
		\begin{array}
			[c]{l}
			\max H(f)\\
			f\in [\mathcal{M}^1(X\times Y)]^\times.
		\end{array}
		\right. 
	\end{equation*}
	The normal-like cone (introduced in \cite{BaBl}) associated to $\mathcal{M}^1(X\times Y)$ at $\overline f \in \mathcal{F}$ is a subset of the dual $\mathcal{F}^*$, and is defined by 
	\begin{eqnarray*}\mathcal{T}_{\mathcal{M}^1(X\times Y)}(\overline f)&:=&\bigcap_{n\geq 1} \overline{\textnormal{conv}}^{w^*}\biggl\{d_F \widehat{\gamma}_c(\overline f): 0\leq \widehat{\gamma}_c (\overline f)\leq \frac{1}{n}; \gamma \in \mathcal{M}^1(X\times Y)\biggr\}\\
		&=& \bigcap_{n\geq 1} \overline{\textnormal{conv}}^{w^*}\biggl\{-\widehat{\gamma}_{|F}: 0\leq \widehat{\gamma}_c (\overline f)\leq \frac{1}{n};  \gamma \in \mathcal{M}^1(X\times Y)\biggr \}
	\end{eqnarray*} 
	Notice that $\mathcal{T}_{\mathcal{M}^1(X\times Y)}(\overline f)\neq \emptyset$ if and only if, $\inf\left\{ \widehat{\gamma}_c (\overline f): \gamma \in \mathcal{M}^1(X\times Y) \right\}=0,$ (see \cite[Proposition 5]{BaBl}). 
	
	\begin{lemma} \label{cherprop1} Let $\overline f\in \mathcal{F}$ be such that $\overline f \leq c$. Suppose that $\mathcal{T}_C(\overline f)\neq \emptyset$. Then, we have:
		\begin{eqnarray*}
			\mathcal{T}_{\mathcal{M}^1(X\times Y)}(\overline f) &=&  -\left \{\widehat{\gamma}_{|F}: \int_{X\times Y} c d\gamma =\int_{X\times Y} \overline f d\gamma; \hspace{1mm} \gamma \in \mathcal{M}^1(X\times Y) \right \}.
		\end{eqnarray*}
		Moreover, $0\not \in \mathcal{T}_{\mathcal{M}^1(X\times Y)}(\overline f)$.
	\end{lemma}
	\begin{proof} It is easy to see, using the convexity argument and the definition of the set $\mathcal{T}_{\mathcal{M}^1(X\times Y)}(\overline f)$, that
		\begin{eqnarray*} 
			\mathcal{T}_{\mathcal{M}^1(X\times Y)}(\overline f) &\subseteq & -\bigcap_{n\geq 1} \overline{\left\{ \widehat{\gamma}_{|F}: 0\leq \widehat{\gamma}_c (\overline f)\leq \frac{1}{n}; \hspace{1mm} \gamma \in \mathcal{M}^1(X\times Y)) \right \}}^{w^*}.
		\end{eqnarray*}
		Set $A:=\left\{ \widehat{\gamma}_{|F}: 0\leq \widehat{\gamma}_c (\overline f) \leq \frac{1}{n}; \hspace{1mm} \gamma \in \mathcal{M}^1(X\times Y) \right \}$. We prove that $A$ is weak$^*$-closed in $\mathcal{F}^*$.  Let $\sigma \in \overline{A}^{w^*}$, then there exists a net $(\gamma_\alpha)_\alpha \subseteq \mathcal{M}^1(X\times Y)$, such that $0\leq \widehat{\gamma}_c (\overline f):=\int_{X\times Y} c d\gamma_\alpha - \int_{X\times Y} \overline fd\gamma_\alpha \leq \frac{1}{n}$ and $(\widehat{\gamma_\alpha}_{|F})$ weak$^*$ converges to $\sigma$. Since $\mathcal{M}^1(X\times Y) $ is weak$^*$-compact, there exists a subnet $(\gamma_\beta)_\beta$ weak$^*$ converging to some $\gamma^*\in \mathcal{M}^1(X\times Y) $. Thus, $((\widehat{\gamma}_\beta)_{|F})$ weak$^*$ converges to $\widehat{\gamma^*}_{|F}=\sigma$ in $\mathcal{F}^*$. Finally, we obtain that $\sigma =\widehat{\gamma^*}_{|F}$, $\widehat{\gamma^*}\in \mathcal{M}^1(X\times Y)$ and $0\leq \int_{X\times Y} c d\gamma^* - \int_{X\times Y} \overline f d\gamma^* \leq \frac{1}{n}$. In other words, $\sigma \in A$, so $A$ is weak$^*$-closed in $\mathcal{F}^*$. Hence, we obtain
		\begin{eqnarray*} 
			\mathcal{T}_{\mathcal{M}^1(X\times Y)}(\overline f) &\subseteq & -\bigcap_{n\geq 1} \left\{ \widehat{\gamma}_{|F}: 0\leq \int_{X\times Y} c d\gamma -\int_{X\times Y} \overline f d\gamma \leq \frac{1}{n}; \hspace{1mm} \gamma \in \mathcal{M}^1(X\times Y)) \right \}\\
			&=& - \left \{ \widehat{\gamma}_{|F}: \int_{X\times Y} c d\gamma =\int_{X\times Y} \overline f d\gamma; \hspace{1mm} \gamma \in \mathcal{M}^1(X\times Y)\right \}.
		\end{eqnarray*}
		The inverse inclusion is trivial. Finally, $0\not \in \mathcal{T}_{\mathcal{M}^1(X\times Y)}(\overline f)$ since $1\in \mathcal{F}$ and $\widehat{\gamma}(1)=1$ for every $\gamma \in \mathcal{M}^1(X\times Y)$.
	\end{proof}
	
	\begin{theorem} \label{thm2} Let $X$ and $Y$ be two completely regular Hausdorff spaces. Let $\mu\in \mathcal{P}_\sigma(X)$, $\nu \in \mathcal{P}_\sigma(Y)$ be two Baire probability measures and $c\in lsc_{bi}(X\times Y)$ such that $v_{\min}{\bf (OP)} <+\infty$. Then, the assertions $(i)$, $(ii)$ and $(iii)$ are equivalent. 
		
		$(i)$ ${\bf (DP)}$ has a solution.
		
		$(ii)$ There exists $(\phi^*,\psi^*, \gamma^*)\in C_b(X)\times C_b(Y) \times \Gamma(\mu, \nu)$ such that $\phi^*\oplus \psi^*\leq c$ and $$\int_{X\times Y} c d\gamma^*=\int_X \phi^* d\mu+ \int_Y \psi^* d\nu.$$ 
		
		$(iii)$ ${\bf (DP)}$ and ${\bf (OP)}$ have solutions and the duality $v_{\max}{\bf (DP)}=v_{\min}{\bf(OP)}$ holds. 
		
		Moreover, if the part $(ii)$ is satisfied, then $(\phi^*,\psi^*)$ is a solution of ${\bf (DP)}$ and $\gamma^*$ is a solution of ${\bf (OP)}$. 
	\end{theorem}

	\begin{proof} The part $(iii) \Longrightarrow (i)$ is trivial. The part $(ii) \Longrightarrow (iii)$ is clear since, $v_{max}{\bf (DP)}\geq \int_X \phi^* d\mu +\int_Y \psi^* d\nu =\int_{X\times Y} c d\gamma^* \geq v_{\min}{\bf(OP)} \geq v_{max}{\bf (DP)}$. Hence, $(\phi^*,\psi^*)$ is a solution of ${\bf (DP)}$ and $\gamma^*$ is a solution of ${\bf (OP)}$ and that the duality $v_{\max}{\bf (DP)}=v_{\min}{\bf(OP)}$ holds. We are going to prove the part $(i) \Longrightarrow (ii)$, by using the Lagrange multipliers rule in \cite{BaBl} and applied to ${\bf (DP)}$ seen as an optimisation problem with an infinite number of inequality constraints. Since, $(\phi^*,\psi^*)$ is a solution of ${\bf (DP)}$ then $\overline f:=\phi^* \oplus \psi^*$ is also a solution to the equivalent problem 
		\begin{equation*}
			{\bf (\widetilde{DP})}
			\left \{
			\begin{array}
				[c]{l}
				\max H(f)\\
				f\in [\mathcal{M}^1(X\times Y)]^\times.
			\end{array}
			\right. 
		\end{equation*}
		where the notations are those defined in the comments before Lemma \ref{cherprop1}, in particular $H(\phi\oplus \psi)=\int_X \phi d\mu +\int_Y \psi d\nu$ for all $\phi \in C_b(X)$, $\psi \in C_b(Y)$. 
		Notice that $\overline f$ cannot belong to the interior of $[\mathcal{M}^1(X\times Y)]^\times$, otherwise we would have $d_F H(\overline f)=H=0$, which is absurd. Thus, $\mathcal{T}_{\mathcal{M}^1(X\times Y)}(\overline f)\neq \emptyset$ (see \cite[Proposition 5]{BaBl}). According to \cite[Definition 1]{BaBl}, it is not difficult to see that $[\mathcal{M}^1(X\times Y)]^\times$ is a weak-admissible set. Thus, using the first-order necessary condition of optimality in \cite[Theorem 1]{BaBl}, there exists $\alpha\geq 0, \beta \geq 0$, $(\alpha, \beta)\neq (0,0)$ and a Lagrange multiplier $L\in \mathcal{T}_{\mathcal{M}^1(X\times Y)}(\overline f)$ satisfying: $ \alpha H+ \beta L=0$. By Lemma \ref{cherprop1}, there exists $\gamma^*\in \mathcal{M}^1(X\times Y)$ such that $L=-\widehat{\gamma^*}_{|F}$ and $\int_{X\times Y} c d\gamma^* =\int_{X\times Y} \overline f d\gamma^*$.  Since $0\not \in \mathcal{T}_{\mathcal{M}^1(X\times Y)}(\overline f)$, then $L\neq 0$ and so $\alpha \neq 0$ (otherwize $(\alpha, \beta)=(0,0)$ a contradiction). Dividing if necessary by $\alpha$, we can assume that $\alpha=1$. In other words, we have $\beta \widehat{\gamma^*}_{|F}=H$. Now, since $H(1_X\oplus 1_Y)=2$, $1_X\oplus 1_Y=2_{X\times Y}$ and $\widehat{\gamma^*}_{|F}(1_X\oplus 1_Y)=\widehat{\gamma^*}(2_{X\times Y})=2$, we get that $\beta=1$ (for a set $Z$, $1_Z$ and $2_Z$ denote the constant functions equal to $1$ and $2$ on $Z$, respectively). Thus, we have $\widehat{\gamma^*}_{|F}=H$ for some $\gamma^* \in \mathcal{M}^1(X\times Y)$ such that $\int_{X\times Y} c d\gamma^* =\int_{X\times Y} \overline f d\gamma^*$. This shows in particular that $\gamma^*\in \Gamma(\mu, \nu)$. Since  $\overline f=\phi^*\oplus \psi^*$, we have 
		
		\begin{eqnarray} \label{eqLO1}
			\int_{X\times Y} c d\gamma^* = \int_{X\times Y} \overline f d\gamma^*=\widehat{\gamma^*}_{|F}(\overline f) &=&H(\phi^*\oplus \psi^*)\nonumber\\
			&=&\int_X \phi^* d\mu +\int_Y \psi^* d\nu \nonumber\\
			&=& v_{\max}{\bf (DP)}.
		\end{eqnarray}
		Hence, $(\phi^*, \psi^*, \gamma^*)$ satisfies  $(ii)$. Thus, $(i) \Longrightarrow (ii)$ is proved. Finally, we see that $\gamma^*$ is necessarily a solution of ${\bf (OP)}$. Indeed, for each $\lambda \in \Gamma(\mu,\nu)$, we have 
		\begin{eqnarray*}
			\int_{X\times Y} c d\lambda &\geq& \int_{X\times Y} \phi^*\oplus \psi^* d\lambda \hspace{2mm} (\textnormal{since } c\geq \phi^*\oplus \psi^* \textnormal{ and } \lambda \textnormal{ is non-negative})\\
			&=& \int_X \phi^* d\mu+\int_Y \psi^* d\nu \hspace{2mm} (\textnormal{since } \lambda \in \Gamma(\mu,\nu))\\
			&=& v_{\max}{\bf (DP)} \hspace{2mm} (\textnormal{since } (\phi^*, \psi^*) \textnormal{ is a solution of the the dual problem} ) \\
			&=& \int_{X\times Y} c d\gamma^* \hspace{2mm} (\textnormal{by } (\ref{eqLO1})).
		\end{eqnarray*}
		This concludes the proof.
		
	\end{proof}		
	
	We now give the main consequence of our previous results, namely that the two problems {\bf (DP)} and {\bf (OP)} have solutions and satisfy duality as soon as the measures $\mu$ and $\nu$ are separable. This generalises the classical case of duality between {\bf (DP)} and {\bf (KP)}, but we will see this in the next section.
	
		\begin{corollary} \label{thm21} Let $X$ and $Y$ be two completely regular Hausdorff spaces. Let $c: X\times Y\to \R$ be a bounded function satisfying the property $(H)$ and $\mu\in \mathcal{P}_s(X)$, $\nu \in \mathcal{P}_s(Y)$ be two separable probability Baire measures. Then, ${\bf (OP)}$ has a solution $\gamma^*\in \Gamma(\mu,\nu)$ and ${\bf (DP)}$ has a solution of the form $(\phi^*,\psi^*):=(\xi^{c \bar c}_0,\xi^c_0)\in C_b(X)\times C_b(Y)$ (where $\xi_0 \in C_b(X)$) satisfying the duality $$\int_X \phi^* d\mu+\int_Y \psi^* d\nu=v_{\max}{\bf (DP)}=v_{\min}{\bf(OP)}=\int_{X\times Y} c d\gamma^*.$$ 
		
	\end{corollary}	
	\begin{proof}
		A consequence of Corollary \ref{Dep} and Theorem \ref{thm2}.
	\end{proof}
	
	\section{Monge-Kantorovich's classical duality as a special case.} \label{S4}
	
	The aim of this section is to show that the classical problem {\bf (KP)} is a special case of the problem {\bf (OP)}. Thus, the duality obtained in \cite[Theorem 1.39]{Sfilippo} and \cite[Theorem 5.10]{Villani} can easily be deduced from the results we have shown in the previous sections.
	
		\begin{lemma} \label{lemmaci}  Let $X$ and $Y$ be perfectly normal Hausdorff spaces. 
		Let $\mu$ and $\nu$	be two Borel probability measures on $X$ and $Y$ respectively and $\gamma \in \Gamma(\mu, \nu)$. Then, for every closed (resp. open) subset $C$ of $X$ and every closed (resp. open) subset $D$ of $Y$, we have $\gamma(C\times Y)=\mu(C)$ and $\gamma(X\times D)=\nu(D)$. If moreover, $\gamma$ is assumed to be a Borel probability measure, then $(\pi_X)_{\#} \gamma=\mu$ on $\mathcal{B}(X)$ and $(\pi_Y)_{\#} \gamma= \nu$ on $\mathcal{B}(Y)$.
	\end{lemma}
	\begin{proof}  Recall that in perfectly normal Hausdorff spaces, Baire and Borel sets coincide and each closed set is a zero-set. Recall also that, every (inner) regular set function is also outer regular (\cite[Lemma 12.3, p. 435]{AB}) and every Borel measure on a perfectly normal space is	(inner) regular \cite[Corollary 7.1.9]{Bvi}. 
		We begin by showing that for every closed set $C$ of $X$, we have $\gamma(C\times Y)=\mu(C)$. Let $V$ be an arbitrary open set such that $C\subseteq V$. By the complete regularity of $X$, there exists a continuous function $\phi : X\to [0,1]$ satisfying $\phi(x)=1$ for each $x\in C$ and $\phi(x)=0$ for all $x\in X\setminus V$. From $\chi_C\leq \phi \leq \chi_{V}$, we have $\chi_{C\times Y}\leq \phi\circ \pi_X \leq \chi_{V\times Y}$ (where, $\chi_A$ denotes the characteristic function of a set $A$). It follows that $$\gamma(C\times Y)=\int_{X\times Y} \chi_{C\times Y}d\gamma \leq \int_{X\times Y} \phi\circ \pi_X d\gamma =\int_X \phi d\mu \leq \int_X \chi_{V} d\mu=\mu(V).$$
		Thus, from the fact that $\mu(C)=\inf\{\mu(V): V \textnormal {open and } C\subset V\}$ (since $\mu$ is also outer regular), we get $\gamma(C\times Y)\leq \mu(C)$. On the other hand, since $\gamma$ is (inner) regular set function, it is also outer regular. Thus, for every $\varepsilon >0$, there exists an open $V_\varepsilon$ of $X$ such that $C\subset V_\varepsilon$ and $\gamma(V_\varepsilon\times Y)\leq \gamma(C\times Y)+\varepsilon$. As above, there exists a continuous function $\phi_\varepsilon : X\to [0,1]$ such that $\chi_C\leq \phi_\varepsilon \leq \chi_{V_\varepsilon}$ and $\chi_{C\times Y}\leq \phi_\varepsilon\circ \pi_X \leq \chi_{V_\varepsilon \times Y}$. Then, we have: 
		\begin{eqnarray*}
		\mu(C)=\int_X \chi_{C}d\mu \leq \int_X \phi_\varepsilon d\mu =\int_{X\times Y} \phi_\varepsilon\circ \pi_X d\gamma &\leq& \int_{X\times Y} \chi_{V_\varepsilon\times Y} d\gamma\\
		&=&\gamma(V_\varepsilon\times Y)\\
		&\leq& \gamma(C\times Y)+\varepsilon.
		\end{eqnarray*}
		It follows that, $\mu(C) \leq \gamma(C\times Y)$, by taking the limit when $\varepsilon \to 0$. Hence, $\gamma(C\times Y)=\mu(C)$ for every closed set $C$ of $X$. Since $\gamma(X\times Y)=1$ and $\mu(X)=1$, we obtain that $\gamma((X\setminus C)\times Y)=\mu(X\setminus C)$ for every closed subset $C$ of $X$. Similarly, we obtain the analogous result for $\nu$. Finally, if moreover $\gamma \in \Gamma(\mu, \nu)$ is assumed to be a Borel probability on $X\times Y$, then using the above we see that the Borel measures $(\pi_X)_{\#} \gamma$ (resp. $(\pi_Y)_{\#} \gamma$) and $\mu$ (resp. $\nu$) coincides on open sets. Hence, $(\pi_X)_{\#} \gamma=\mu$ on $\mathcal{B}(X)$ and $(\pi_Y)_{\#} \gamma= \nu$ on $\mathcal{B}(Y)$, by \cite[Lemma 7.1.2]{Bvi}.
		
	\end{proof}
	\begin{theorem} \label{lemma3} Let $X$ and $Y$ be perfectly normal Hausdorff spaces. Suppose that $\mu\in \mathcal{P}_r(X)$ and $\nu \in \mathcal{P}_r(Y)$ are two Radon probability measures. Then, $\Gamma(\mu,\nu)=\Pi(\mu,\nu)$. In consequence, the problems ${\bf(OP)}$ and ${\bf(KP)}$ coincide. Thus, if $c\in lsc_{bi}(X\times Y)$ is such that $v_{\min}{\bf(KP)} <+\infty$, then ${\bf(KP)}$ has a solution.
	\end{theorem}
	\begin{proof} We will use the characterization of Radon measures given in \cite[Theorem 7.10.6]{Bvi}. Let $\gamma \in \Gamma(\mu, \nu)$. Fix $\varepsilon >0$ and find two compact sets $K_\varepsilon \subseteq X$ and $L_\varepsilon \subseteq Y$ such that $\mu(X\setminus K_\varepsilon)<\frac{\varepsilon}{2}$ and $\nu(Y\setminus L_\varepsilon)<\frac{\varepsilon}{2}$ (this is possible, since $\mu$ and $\nu$ are Radon probability measures by assumption). We have  $K_\varepsilon \times L_\varepsilon \in \mathcal{U}(X\times Y)$ and by using Lemma \ref{lemmaci},
		\begin{eqnarray*} \label{eqbargamma} \gamma((X\times Y)\setminus (K_\varepsilon \times L_\varepsilon)) &\leq& \gamma\left ( (X\setminus K_\varepsilon) \times Y\right)+\gamma\left ( X\times (Y\setminus L_\varepsilon) \right)\nonumber \\
			&=& \mu(X\setminus K_\varepsilon)+\nu(Y\setminus L_\varepsilon)\\
			&<& \varepsilon.
		\end{eqnarray*}
		Now, since the function $L: f \mapsto \int_{X\times Y} f d\gamma$ is linear continuous on $C_b(X\times Y)$ and since for every $f\in C_b(X\times Y)$ such that $f_{|K_\varepsilon \times L_\varepsilon}=0$, we have 
		\begin{eqnarray*}
			L(f):=\int_{X\times Y} f d \gamma&=&\int_{(X\times Y)\setminus (K_\varepsilon \times L_\varepsilon)} f d \gamma\\
			&\leq& \gamma((X\times Y)\setminus (K_\varepsilon \times L_\varepsilon))\|f\|_{\infty}\\
			&\leq& \varepsilon \|f\|_{\infty},
		\end{eqnarray*}
		we get, using \cite[Theorem 7.10.6]{Bvi}, that $\gamma$ is a Radon measure on $X\times Y$. Hence, $\gamma \in \Pi(\mu, \nu)$, that is, $\Gamma(\mu, \nu)\subseteq \Pi(\mu,\nu)$. The reverse inclusion is always true. In consequence, the problems ${\bf(OP)}$ and ${\bf(KP)}$ coincide and by Proposition \ref{propOPS} have also solutions.
	\end{proof}
	
	From our previous results, we obtain as an immediate consequence the classical Monge-Kantorovich's duality in perfectly normal Hausdorff spaces.
	
		\begin{corollary} \label{thm31} Let $X$ and $Y$ be two perfectly normal Hausdorff spaces. Let $c: X\times Y\to \R$ be a bounded function satisfying the property $(H)$ and $\mu\in \mathcal{P}_r(X)$, $\nu\in \mathcal{P}_r(Y)$ be two Radon probability measures. 
		Then, ${\bf (KP)}$ has a solution $\gamma^*\in \Pi(\mu,\nu)$ and ${\bf (DP)}$ has a solution of the form $(\phi^*,\psi^*):=(\xi^{c \bar c}_0,\xi^c_0)\in C_b(X)\times C_b(Y)$ (where $\xi_0 \in C_b(X)$) satisfying the duality $$\int_X \phi^* d\mu+\int_Y \psi^* d\nu=v_{\max}{\bf (DP)}=v_{\min}{\bf(KP)}=\int_{X\times Y} c d\gamma^*.$$ 
		In particular, $\phi^*\oplus\psi^* =c$, $\gamma^*$-a.e.
	\end{corollary}	
	\begin{proof}
		The first part is a direct consequence of Corollary \ref{thm21} and Theorem \ref{lemma3}. The fact that $\phi^*\oplus\psi^* =c$, $\gamma^*$-a.e, is a consequence of the formula $\int_X \phi^* d\mu+\int_Y \psi^* d\nu=\int_{X\times Y} c d\gamma^*$. Indeed, since $\gamma^*\in \Pi(\mu,\nu)$, then $\gamma^*$ is a Radon probability measure satisfying the equality $\int_X \phi^* d\mu+\int_Y \psi^* d\nu=\int_{X\times Y} \phi^*\oplus \psi^* d\gamma^*$. Then, we get that $\int_{X\times Y} \left(c -\phi^*\oplus \psi^*\right) d\gamma^* =0$. Since, $c -\phi^*\oplus \psi^* \geq 0$, we obtain that $\phi^*\oplus\psi^* =c$, $\gamma^*$-a.e.
	\end{proof}
	
\begin{remark} As mentioned in the introduction, our approach in this article did not use the concept of $c$-cyclic monotonicity. However, the fact that a solution of the problem {\bf (KP)} is concentrated on a $c$-cyclically monotone set can be trivially deduced from Corollary \ref{thm31}. Recall that a set $\Lambda \subseteq X\times Y$ is said to be  $c$-cyclically monotone if for all $k \in \N$ and $\{(x_i, y_i)\}_{i=1}^k \subseteq \Lambda$ and for any permutation $\sigma \in \mathcal{S}_k$ 
$$\sum_{i=1}^k c(x_i, y_i) \leq \sum_{i=1}^k c(x_i, y_{\sigma(i)}).$$
For valuable details on the concept of $c$-cyclic monotonicity, we refer to	\cite{DKW} and the references therein.

Let $X$ and $Y$ be two perfectly normal Hausdorff spaces. Let $c: X\times Y\to \R$ be a bounded function satisfying the property $(H)$ and $\mu\in \mathcal{P}_r(X)$, $\nu\in \mathcal{P}_r(Y)$ be two Radon probability measures. Suppose that $\gamma \in \Pi(\mu, \nu)$ is a solution of ${\bf (KP)}$. Then, $\gamma$ is concentrated on a set $\Lambda\subseteq X \times Y$ which is $c$-cyclically monotone. Indeed, by Corollary \ref{thm31}, there exists $(\phi^*,\psi^*):=(\xi^{c \bar c}_0,\xi^c_0)\in C_b(X)\times C_b(Y)$ (where $\xi_0 \in C_b(X)$) such that $\phi^*\oplus\psi^* =c$, $\gamma$-a.e. Thus, there exists a measurable set $\Lambda\subseteq X \times Y$ such that $\gamma(\Lambda)=1$ and $\Lambda \subseteq \left\{(x,y)\in X \times Y: \phi^*(x)+\psi^*(y)=c(x,y)\right\}$. This implies that $\Lambda$ is $c$-cyclically monotone.
\end{remark}

	If we are only interested in weak duality, i.e. the equality of $v_{\max}{\bf (DP)}=v_{\min}{\bf(KP)}$ without worrying about the existence of a solution for ${\bf (DP)}$, Corollary \ref{thm31} can be extended to more general condition on the cost $c$. In a metric space, this would mean assuming that $c$ is bounded from below and lower semicontinuous. 
	Recall that if $X$ and $Y$ are metric spaces, every non-negative proper lower semicontinuous $c: X\times Y\to \R\cup\{+\infty\}$ can be writen $$c(x,y)=\lim_{n\to+\infty} c_n(x,y),$$ 
	where, $(c_n)_{n\in \N}$ is a nondecreasing sequence of bounded functions having the property $(H)$. To see this, just choose
	$$c_n(x,y)=\inf_{(z,t)\in X\times Y} \left\{\min\left(c(z,t),n\right) +n\left[ d(x,z)+d(y,t)\right]\right\}.$$
	In fact, $c_n$ is $n$-Lipschitz ensuring the property $(H)$, is also nondecreasing in $n$, and further satisfies $0\leq c_n\leq \min\left(c(x,y),n\right)$ for each $n\in \N$ and each $(x,y)\in X\times Y$. The condition on $c$ in the following result is a natural extension of the property we have just described to perfectly normal Hausdorff spaces.
	
	The result that follows is well known (see \cite[Theorem 5.10]{Villani} in Polish spaces and \cite{Eda, Khg} in completely regular spaces), but we present it here with an alternative proof that builds on the previous results.
	
	\begin{theorem} \label{thm111} Let $X$ and $Y$ be two perfectly normal Hausdorff spaces, $\mu\in \mathcal{P}_r(X)$ and $\nu\in \mathcal{P}_r(Y)$ be two Radon probability measures. Let $c\in lsc_{bi}(X\times Y)$ be such that $v_{\min}{\bf (KP)} <+\infty$. Suppose that $c(x,y)=\lim_{n\to+\infty} c_n(x,y)$, where $(c_n)_{n\in \N}$ is a nondecreasing sequence of bounded  non-negative functions having the property $(H)$. Then, the duality $v_{\max}{\bf (DP)}=v_{\min}{\bf (KP)}$ holds.
	\end{theorem}
	
	\begin{proof}  For all $n\in \N$, define  $$F(c_n):=\{(\phi,\psi)\in C_b(X)\times C_b(Y): \phi(x)+\psi(y)\leq c_n(x,y); \forall (x,y)\in X\times Y\},$$ 
		and let us set
		$$F(c):=\{(\phi,\psi)\in C_b(X)\times C_b(Y): \phi(x)+\psi(y)\leq c(x,y); \forall (x,y)\in X\times Y\},$$
		the admissible set for ${\bf (DP)}$. Clearly, $F (c_n)\subset F (c)$ for every $n\in \N$ and we have,
		\begin{eqnarray*} \label{DPn}
			v_{\max}{\bf (DP)_n}&:=&\sup \left \{ \int_X \phi d\mu+\int_Y \psi d\nu: (\phi, \psi)\in F (c_n) \right\} \leq v_{\max}{\bf (DP)}.
		\end{eqnarray*}	
		\begin{eqnarray*} 
			v_{\min}{\bf (KP)_n}&:=&\inf\left \{\int_{X\times Y} c_n d\gamma: \gamma \in \Pi(\mu,\nu)\right \}\leq v_{\min}{\bf (KP)}.
		\end{eqnarray*} 
		By Corollary \ref{thm31} applied with the cost $c_n$ which is a bounded function having the propert $(H)$, both the problems ${\bf (KP)_n}$ and ${\bf (DP)_n}$ have solutions satisying the following duality: 
		\begin{eqnarray*} \label{Eqfin}
			v_{\min}{\bf(KP)_n}=\int_{X\times Y} c_n d\gamma_n =\int_X \phi^*_n d\mu+\int_Y \psi^*_n d\nu=v_{\max}{\bf (DP)_n},
		\end{eqnarray*}
		where $\gamma_n$ is a solution of ${\bf (KP)_n}$ and $(\phi^*_n, \psi^*_n)$ is a solution of ${\bf (DP)_n}$. Then, we have, 
		\begin{eqnarray} \label{eqA}
			v_{\min}{\bf (KP)_n}:=\min_{\gamma \in \Pi(\mu, \nu)}\int_{X\times Y} c_n d\gamma=v_{\max}{\bf (DP)_n}\leq v_{\max}{\bf (DP)}\leq v_{\min}{\bf (KP)}.
		\end{eqnarray}
		To concludes that the duality holds, it suffices to prove that $\lim_{n\to +\infty}v_{\min}{\bf(KP)_n}=v_{\min}{\bf(KP)}$.	The set $\Pi(\mu, \nu)$ is weak-compact (see Theorem ~\ref{lemma3} and Proposition ~\ref{propOPS}) and for each $n\in \N$, the function $\widehat{c}_n: \gamma \mapsto \int_{X\times Y} c_n d\gamma$ is weak-continuous on $\Pi(\mu, \nu)$ with $\widehat{c}_n(\gamma)\leq \widehat{c}_{n+1}(\gamma)$ for $n\in \N$ and each $\gamma \in \Pi(\mu, \nu)$. Using the limit and minimum interversion result in \cite[Corollary p. 407]{Terk}, we have 
		\begin{eqnarray} \label{eqeq5}
			\min_{\gamma \in \Pi(\mu, \nu)} \lim_{n\to +\infty} \int_{X\times Y} c_n d\gamma =\lim_{n\to +\infty} \min_{\gamma \in \Pi(\mu, \nu)} \int_{X\times Y} c_n d\gamma.
		\end{eqnarray}
		By the monotone convergence theorem, we have $\lim_{n\to +\infty} \int_{X\times Y} c_n d\gamma =\int_{X\times Y} c d\gamma$, for each $\gamma \in \Pi(\mu, \nu)$. Using $(\ref{eqeq5})$, we get $v_{\min}{\bf (KP)}=\lim_{n\to +\infty}v_{\min}{\bf(KP)_n}$. Finally, from $(\ref{eqA})$, we obtain the duality $v_{\min}{\bf (KP)}=v_{\max}{\bf (DP)}$.
	\end{proof}
	
	\paragraph{\bf An example in finite dimension under the Solgenfrey topology.} 
	In finite dimension, we can relax the conditions on the cost function $c$ by assuming that it only satisfies the property $(H)$ for a topology finer than the usual topology. In Example \ref{ex2} below, we will restrict ourselves to the Euclidean spaces $\R^n$ and $\R^p$, $n, p \geq 1$ and assume that $c$ satisfies the property $(H)$ when $\R^n \times \R^p$ is provided with the Sorgenfrey product topology which is weaker than assuming the property $(H)$ for the usual product topology on $\R^n \times \R^p$. This makes it possible to extend cost functions to a class of functions that are not necessarily even lower semicontinuous for the usual topology. In this case we can guarantee both duality and the existence of solutions for ${\bf (DP)}$ in a space containing $C_b(\R^n)\times C_b(\R^p)$ but strictly included in $\mathcal{L}^1(\mu)\times \mathcal{L}^1(\nu)$.
	
	Recall that the Sorgenfrey $n$-space is defined as the space $\R^n$ with the finer topology denotes $\tau_n$ whose base consists of all rectangles $\prod_{i=1}^n [a_i, b_i)$, where $a_i< b_i$, $i=1,...,n$ are real numbers. We denote $\R^n_\ell$ the space $\R^n$ equipped with this topology. It is well-known that this topology generates the same Borel $\sigma$-algebra as the Euclidean topology of $\R^n$, that is $\mathcal{B}(\R^n_\ell)=\mathcal{B}(\R^n)$. It is also well-known that $\R^n_\ell$ is a separable completely regular, but neither normal nor metrizable space. It follows that every Borel (equivalently, Radon) probability measure on $\R^n$ is a separable Borel probability measure on $\R^n_\ell$ (see \cite[Example 6.1.19 \& Example 7.2.4]{Bvi}), but is not Radon in general on $\R^n_\ell$ since compact subsets of $\R^n_\ell$ are at most countable subsets. In fact we have $\mathcal{P}_\sigma (\R^n)=\mathcal{P}_r (\R^n)=\mathcal{P}_s (\R^n_\ell)$. It is easy to see that any bounded function $c: \R^n \times \R^p \to \R$ satisfying the property $(H)$ on $\R^n \times \R^p$ will always satisfy $(H)$ on $\R^n_\ell\times \R^p_\ell$. However, the converse is not always true. A trivial example in $\R\times \R$ (here $n=p=1$), is to take $c(x,y)=\chi_{\R^+}(x)+\chi_{\R^+}(y)$, for all $(x,y)\in \R\times \R$ , where $\chi_{\R^+}$ denotes the characteristic function of $\R^+$. It is thus easy to see that $c$ is not lower semicontinuous for the usual topology of $\R\times \R$, however it is continuous and satisfies the property $(H)$ on $\R_\ell\times \R_\ell$.
	
	\begin{example} \label{ex2} Let $\mu\in \mathcal{P}_r(\R^n)$ and $\nu \in \mathcal{P}_r(\R^p)$ be two Radon probability measures. Suppose that $c: \R^n_\ell\times \R^p_\ell\to \R$ be a bounded function satisfying the property $(H)$. Then, ${\bf (KP)}$ has a solution $\gamma^* \in \Pi(\mu,\nu)$ and ${\bf (DP)}$ under the admissible set $\left\{(\phi,\psi)\in C_b(\R^n_\ell)\times C_b(\R^p_\ell): \phi\oplus \psi \leq c\right\}$, has a solution of the form $(\phi^*,\psi^*):=(\xi^{c \bar c}_0,\xi^c_0)\in C_b(\R^n_\ell)\times C_b(\R^p_\ell)$ (where $\xi_0 \in C_b(\R^n_\ell)$). Moreover the duality $v_{\max}{\bf (DP)}=v_{\min}{\bf(KP)}$, holds.	
	\end{example}
	
	\begin{proof}
		By the above remarks, $\mu\in \mathcal{P}_s (\R^n_\ell)$ and $\nu\in \mathcal{P}_s (\R^p_\ell)$ and by assumption $c: \R^n_\ell\times \R^p_\ell\to \R$ is a bounded function satisfying the property $(H)$. Using Corollary \ref{thm21}, the problem ${\bf (OP)}$ has a solution $\gamma^*\in \Gamma(\mu,\nu)$, ${\bf (DP)}$ has a solution of the form $(\phi^*,\psi^*):=(\xi^{c \bar c}_0,\xi^c_0)\in C_b(\R^n_\ell)\times C_b(\R^p_\ell)$ (where $\xi_0 \in C_b(\R^n_\ell)$). Moreover the duality $v_{\max}{\bf (DP)}=v_{\min}{\bf(OP)}$ holds. Now, we prove that the set $\Gamma(\mu,\nu)$ defined on $\R^n_\ell \times \R^p_\ell$ is a subset of the transport plan $\Pi(\mu, \nu)$ defined on $\R^n \times \R^p$. Indeed, let $\gamma \in \Gamma(\mu, \nu)$, then $\gamma \in \mathcal{M}^1(\R^n_\ell\times \R^p_\ell)$ and $\int_{\R^n\times \R^p} \phi\oplus \psi d\gamma =\int_{\R^n} \phi d\mu+\int_{\R^p} \psi d\nu$, for every $\phi \in C_b(\R^n_\ell)$ and $\psi \in C_b(\R^p_\ell)$. Since the Sorgenfrey topology is finer than the usual topology, we have $\mathcal{M}^1(\R^n_\ell\times \R^p_\ell)\subseteq \mathcal{M}^1(\R^n\times \R^p)$, $C_b(\R^n)\subseteq C_b(\R^n_\ell)$ and $C_b(\R^p)\subseteq C_b(\R^p_\ell)$. It follows that $\gamma \in \mathcal{M}^1(\R^n\times \R^p)$ and $\int_{\R^n\times \R^p} \phi\oplus \psi d\gamma =\int_{\R^n} \phi d\mu+\int_{\R^p} \psi d\nu$, for every $\phi \in C_b(\R^n)$ and $\psi \in C_b(\R^p)$. Since $\mu$ and $\nu$ are Radon on $\R^n$ and $\R^p$ equipped with there usual topologies, using Theorem \ref{lemma3}, we get that $\gamma$ is a Radon probability measure on $\R^n \times \R^p$ for the usual product topology. Hence, $\gamma \in \Pi(\mu, \nu)$ and so $\Gamma(\mu, \nu)\subseteq \Pi(\mu, \nu)$. It follows that $v_{\min}{\bf(KP)}\leq v_{\min}{\bf(OP)}$. On the other hand, we know that $v_{\min}{\bf(DP)}\leq v_{\min}{\bf(KP)}$ and $v_{\max}{\bf (DP)}=v_{\min}{\bf(OP)}=\int_{\R^n\times \R^p} c d\gamma^*$. Thus, $v_{\min}{\bf(DP)}= v_{\min}{\bf(KP)}=\int_{\R^n\times \R^p} c d\gamma^*$. That is $\gamma^*$ is also a solution of ${\bf(KP)}$ and that the duality holds.
	\end{proof}

	\section*{Acknowledgement}
	
	This research has been conducted within the FP2M federation (CNRS FR 2036) and  SAMM Laboratory of the University Paris Panthéon-Sorbonne.
	
	\section*{Declaration} 
	
	- The authors declare that there is no conflict of interest.
	
	- Data sharing not applicable to this article as no datasets were generated or analysed during the current study.
	
	\bibliographystyle{amsplain}

\begin{thebibliography}{999}		
		\bibitem{AB} C. D. Aliprantis, K. C. Border, Infinite dimensional analysis A hitchhiker’s guide (3rd Edi-
		tion), Springer (2006).
		%
		\bibitem{BaBl} M. Bachir, J. Blot, {\it Lagrange multipliers in locally convex spaces}, J. Optim. Theory Appl. 201(3), (2024) 1275-1300.
		%
		\bibitem{Ba1} M. Bachir, {\it Convex extension of lower semicontinuous functions defined on normal Hausdorff space}, J. Convex Anal. 27(3) (2020) 1033-1049.
		%
		\bibitem{BPS} M. Beiglböck, G. Pammer, S. Schrott, {\it Denseness of biadapted Monge mappings}, Ann. Inst. Henri Poincaré, Probab. Stat. 61(1) (2025) 329-349.
		%
		\bibitem{BSch} M. Beiglböck, W. Schachermayer, {\it Duality for Borel measurable cost functions}, Trans. Am. Math. Soc. 363(8) (2011) 4203-4224.
		%
		%
		\bibitem{Bvi} V. I. Bogachev, Measure Theory, Vol. II, Springer, Berlin, 2007.
		%
		\bibitem{DKW} L. De Pascale, A. Kausamo, K Wyczesany, {\it 60 years of cyclic monotonicity: a survey}, J. Convex Anal. 32 (2), (2025) 399-430.
		%
		\bibitem{Drm} R. M. Dudley, {\it Convergence of Baire measures}, Studia Math., 27 (1966), 251-268.
		%
		\bibitem{Eda} D. A. Edwards, {\it A simple proof in Monge-Kantorovich duality theory}, Studia Math., 200 (1), (2010) 67-77. 
		%
		\bibitem{Fz} Z. Frol\'ik, {\it The topological product of two pseudocompact spaces}, Czechoslovak Mathematical Journal, 10 (3), (1960) 339-349
		%
		\bibitem{Jj} J. Jahn, {\it Introduction to the Theory of Nonlinear Optimization}, Springer, Berlin (2007).
		%
		\bibitem{Khg} H. G. Kellerer, {\it Duality theorems for marginal problems}, Z. Wahrsch. Verw. Gebiete 67 (1984), 399-432.
		%
		\bibitem{Krb} R. B. Kirk, {\it A note on the Mackey topologyfor $(C_b(X)^*, C_b(X))$}, Pacific J. Math. 45 (1973), 543-554.
		%
		\bibitem{Sfilippo} F. Santambrogio, {\it Optimal Transport for Applied Mathematicians: Calculus of Variations, PDEs, and Modeling}, Springer International Publishing, (2015) p. 1-57.
		%
		\bibitem{SwTj} W. Schachermayer and J. Teichmann, {\it Characterization of optimal transport plans for the Monge–Kantorovich problem}, Proc. Amer. Math. Soc. 137 (2009), 519-529.
		%
		\bibitem{Sf} F. D. Sentilles, {\it Bounded continuous functions on a completely regular space}, Trans. Amer. Math. Soc, 168 (1972), 311-336.
		%
		\bibitem{Sd} D. Sentilles and R. F. Wheeler, {\it Linear functional and partitions of unity in $C_b(X)$}, Duke Math. J. 41 (1974), 483-496.
		%
		\bibitem{Sl} L. Schwartz, {\it Analyse : Topologie générale et analyse fonctionnelle}, Hermann (2008).
		%
		\bibitem{St} R. M. Stephenson, Jr., {\it Pseudocompact spaces}, Trans. Amer. Math. Soc. 134(3) (1968) 437-
		448.
		%
		\bibitem{Terk} F. Terkelsen {\it Some minimax theorems}, Math. Scand. 31 (1972), 405-413. 
		%
		\bibitem{Th} H. Tong, {\it Some characterizations of normal and perfectly normal spaces}, Duke Math. J. 19 (1952), 289-292
		%
		\bibitem{Va} V. S. Varadarajan, {\it Measures on topological spaces}, Mat. Sb. 55 (97) (1961), 35-100; English transl., Amer. Math. Soc. Transi. 48 (2) (1965), 161-228.
		%
		\bibitem{Villani0} C. Villani, {\it Topics in Optimal Transportation}, Grad. Stud. Math. 58, Amer. Math. Soc., Providence RI, 2003.
		%
		\bibitem{Villani} C. Villani, {\it Optimal transport: Old and New}, Springer Verlag, 2008.
		%
		\bibitem{Wrf1} R. F. Wheeler,  {\it The strict topology, separable measures, and paracompactness}, Pacific J. Math. 47 (1), (1973), 287-302.
		%
		\bibitem{Wrf} R. F. Wheeler, {\it Weak and pointwise compactness in the space of bounded continuous
			functions}, Trans. Amer. Math. Soc, 266 (2) (1981) 515-530.
	\end{thebibliography}
			
	\end{document}